\documentclass[11pt]{amsart}

\usepackage[utf8]{inputenc}
\usepackage{amsmath, amssymb,amsthm,tikz}
\usepackage{tikz-cd}
\usepackage{tikz}
\usepackage{todonotes}
\usepackage{mathrsfs}
\usepackage{extarrows}
\usepackage{graphicx}
\usepackage{thmtools}
\usepackage{mathtools}
\usepackage{hyperref}
\usepackage{multirow}

\usepackage{nicematrix}
\usepackage{adjustbox}
\usepackage[margin=1.48in]{geometry}

\usepackage[english]{babel}
\usepackage{comment}

\usepackage[toc,page]{appendix}

\newcommand{\Z}{\mathbb{Z}}
\newcommand{\R}{\mathbb{R}}
\newcommand{\C}{\mathbb{C}}
\newcommand{\Q}{\mathbb{Q}}

\newcommand{\mcA}{\mathcal{A}}
\newcommand{\mcB}{\mathcal{B}}
\newcommand{\mcD}{\mathcal{D}}
\newcommand{\mcG}{\mathcal{G}}
\newcommand{\mcJ}{\mathcal{J}}
\newcommand{\mcM}{\mathcal{M}}
\newcommand{\mcS}{\mathcal{S}}
\newcommand{\bbG}{\mathbb{G}}
\newcommand{\bbM}{\mathbb{M}}

\newcommand{\bbN}{\mathbb{N}}

\newcommand{\bbbU}{{\mathbf{U}}}
\newcommand{\bbS}{\mathbb{S}}

\newcommand{\mcE}{\mathcal{E}}
\newcommand{\mcH}{\mathcal{H}}
\newcommand{\mcV}{\mathcal{V}}

\newcommand{\bA}{\mathbf{A}}

\newcommand{\cus}{\mathrm{c}}
\newcommand{\Oo}{\mathcal{O}}

\newcommand{\End}{\operatorname{End}}

\newcommand{\II}{\operatorname{I}}

\newcommand{\ind}{\operatorname{ind}}
\newcommand{\Ind}{\operatorname{Ind}}
\newcommand{\meas}{\operatorname{meas}}
\newcommand{\Mod}{\operatorname{Mod}}
\newcommand{\trace}{\operatorname{tr}}
\newcommand{\Frob}{\mathrm{Fr}}
\newcommand{\rU}{\mathrm{U}}

\newcommand{\bG}{\mathbf{G}}
\newcommand{\bM}{\mathbf{M}}

\newcommand{\bS}{\mathbf{S}}
\newcommand{\bT}{\mathbf{T}}

\newcommand{\ad}{\mathrm{ad}}

\newcommand{\sconn}{\mathrm{sc}}

\newcommand{\tame}{\operatorname{t}}

\newcommand{\aff}{\mathrm{aff}}

\newcommand{\Hom}{\operatorname{Hom}}

\newcommand{\cInd}{\operatorname{c-Ind}}
\newcommand{\Cent}{\operatorname{Z}}

\newcommand{\Nor}{\operatorname{N}}
\newcommand{\fR}{\mathfrak{R}}
\newcommand{\fB}{\mathfrak{B}}
\newcommand{\fgg}{\mathfrak{g}}

\newcommand{\Irr}{\operatorname{Irr}}

\newcommand{\fa}{\mathfrak{a}}
\newcommand{\fff}{\mathfrak{f}}
\newcommand{\fl}{\mathfrak{l}}
\newcommand{\fL}{\mathfrak{L}}
\newcommand{\fo}{\mathfrak{o}}
\newcommand{\fp}{\mathfrak{p}}
\newcommand{\fP}{\mathfrak{P}}
\newcommand{\fs}{\mathfrak{s}}
\newcommand{\fS}{\mathfrak{S}}

\newcommand{\fr}{\mathfrak{r}}
\newcommand{\fe}{\mathfrak{e}}
\newcommand{\fw}{\mathfrak{w}}
\newcommand{\fX}{{\mathfrak{X}}}

\newcommand{\bounded}{{\mathrm{b}}}
\newcommand{\der}{{\mathrm{der}}}
\newcommand{\nr}{{\mathrm{nr}}}
\newcommand{\red}{{\mathrm{red}}}
\newcommand{\scn}{{\mathrm{sc}}}
\newcommand{\enh}{{\mathrm{e}}}
\newcommand{\depth}{{\mathrm{depth}}}

\newcommand{\bbbM}{{\mathbf{M}}}
\newcommand{\bbbS}{{\mathbf{S}}}
\newcommand{\rZ}{{\mathrm{Z}}}

\newcommand{\bkappa}{{\boldsymbol{\kappa}}}

\newcommand{\rankkF}{{\mathrm{r}}}

\newcommand{\rA}{\operatorname{A}}
\newcommand{\Ad}{\operatorname{Ad}}
\newcommand{\rE}{\operatorname{E}}
\newcommand{\rF}{\operatorname{F}}
\newcommand{\Gal}{\operatorname{Gal}}
\newcommand{\sep}{\operatorname{sep}}
\newcommand{\GL}{\mathrm{GL}}

\newcommand{\SL}{\operatorname{SL}}

\newcommand{\rG}{\operatorname{G}}

\newcommand{\Sp}{\operatorname{Sp}}
\newcommand{\SO}{\operatorname{SO}}

\newcommand{\GSp}{\operatorname{GSp}}

\newcommand{\Mat}{\operatorname{Mat}}

\newcommand{\scusp}{\mathrm{sc}}
\newcommand{\cusp}{\mathrm{c}}
\newcommand{\val}{\mathrm{val}}
\hypersetup{
    colorlinks,
    citecolor=blue,
    filecolor=blue,
    linkcolor=blue,
    urlcolor=blue
}

\renewcommand{\tilde}{\widetilde}

\title[Hecke algebras for $p$-adic groups]{Hecke algebras for 
$p$-adic reductive groups and Local Langlands Correspondence for Bernstein blocks
}

\author{Anne-Marie Aubert}
\address{Sorbonne Universit\'e and Universit\'e Paris Cit\'e, CNRS,
IMJ-PRG, F-75005 Paris, France}
\email{anne-marie.aubert@imj-prg.fr}

\author{Yujie Xu}
\address{M.I.T., 77 Massachusetts Avenue,
Cambridge, MA, USA}
\email{yujiexu@mit.edu}

\address{}
\curraddr{}
\email{}
\date{\today}

\numberwithin{equation}{subsection}
\newtheorem{theorem}[equation]{Theorem}
\newtheorem{prop}[equation]{Proposition}
\newtheorem{thm}{Theorem}

\newtheorem{lem}[equation]{Lemma}
\newtheorem{Coro}[equation]{Corollary}

\newtheorem{property}[equation]{Property}
\theoremstyle{definition}
\newtheorem{Defn}[equation]{Definition}
\newtheorem{remark}[equation]{Remark}
\newtheorem{numberedparagraph}[equation]{}
\newtheorem{conj}[equation]{Conjecture}

\begin{document}

\maketitle
\begin{abstract}

We study the endomorphism algebras attached to Bernstein components of reductive $p$-adic groups and construct a local Langlands correspondence with the appropriate set of enhanced $L$-parameters, using certain ``desiderata'' properties for the LLC for supercuspidal representations of proper Levi subgroups. We give several applications of our LLC to various reductive groups with Bernstein blocks cuspidally supported on general linear groups. 

In particular, for Levi subgroups of maximal parabolic of the split exceptional group $\rG_2$, we compute the explicit weight functions for the corresponding Hecke algebras, and show that they satisfy a conjecture of Lusztig's. Some results from $\S \ref{sec:G2}$ are used by the same authors to construct a full local Langlands correspondence in \cite{AX-LLC}. 
Moreover, we also prove a reduction to depth zero case result for the Bernstein components attached to regular supercuspidal representations of Levi subgroups. 
\end{abstract}

\tableofcontents

\section{Introduction}
\addtocontents{toc}{\protect\setcounter{tocdepth}{0}}
\subsection{Background}

Let $F$ be a non-archimedean local field. Let $\bG$ be a connected reductive group defined over $F$, and $G$ its group of $F$-points.  
Let $M$ be the Levi subgroup of a parabolic subgroup $P$ of $G$.  

Let $\fs=[M,\sigma]_G$ be the inertial class attached to the pair $(M,\sigma)$, where $\sigma$ is a supercuspidal irreducible representation of $M$. Recall that this means that $\fs$ is the $G$-conjugacy class of $(M,\fX_\nr(M)\cdot\sigma)$, where $\fX_\nr(M)\cdot\sigma$ is the orbit of $\sigma$ under $\fX_\nr(M)$ --the group of unramified characters of $M$. Let $\fB(G)$ be the set of such $\fs$'s. 
We denote by $\Irr^\fs(G)$ the Bernstein series of irreducible representations of $G$ whose cuspidal support lies in $\fs$ (see $\mathsection$\ref{subsec:general} for a precise definition). 

Let $W_G^\fs$ denote the extended finite Weyl group $\Nor_G(\fs_M)/M$, where $\fs_M=[M,\sigma]_M$, 
and let $W_G^{\fs,x}$ be the stabilizer of $x\in \Irr^{\fs_M}(M)$ in $W_G^\fs$. 
By \cite{Solleveld-endomorphism-algebra}, there exists a collection $(\natural_x)_x$ of $2$-cocycles for $x\in \Irr^{\fs_M}(M)$, 
\begin{equation}
\natural_x\colon W_G^{\fs,x}\times W_G^{\fs,x}\longrightarrow\C^\times,
\end{equation}
such that we have a bijection
\begin{equation} \label{eqn:Sol}
 \xi_G^\fs\colon   \Irr^\fs(G) \longrightarrow (\Irr^{\fs_M}(M)/\!/W_G^\fs)_\natural,
\end{equation}
where $(\Irr^{\fs_M}(M)/\!/W_G^\fs)_\natural$ is a \textit{twisted extended quotient} in the sense of \cite[\S2.1]{ABPS-CM} (see \ref{eqn:teq} for the precise definition).

A parallel picture to (\ref{eqn:Sol}) exists on the Galois side. Let $W_F$ be the absolute Weil group of $F$ and $I_F$ its inertia subgroup. Let $M^\vee$ be the Langlands dual group of $M$, i.e.~it is a complex Lie group with root datum dual to that of $M$. It is equipped with an action of $W_F$, and we write ${}^LM:=M^{\vee}\rtimes W_F$. The group $M^\vee$ acts on the set of \textit{cuspidal} $M$-relevant enhanced $L$-parameters for $M$--a terminology based on Lusztig's notion of cuspidal pairs (see Definition~\ref{def:cuspidal} for more details). Let $\Phi_\enh^\cusp(M)$ be the set of $M^\vee$-conjugacy classes of cuspidal enhanced $L$-parameters for $M$.

Let $\rZ_{M^\vee\rtimes I_F}$ be the center of $M^\vee\rtimes I_F$. The group $\fX_\nr({}^LM):=(\rZ_{M^\vee\rtimes I_F})_{W_F}^\circ$, which is naturally isomorphic to the group $\fX_\nr(M)$ (see \cite[\S3.3.1]{Haines}), acts naturally on the set of cuspidal $M$-relevant enhanced $L$-parameters for $M$.
We denote by $\fs^\vee=[{}^LM,\varphi_\cus,\varrho_\cus]_{G^\vee}$
the $G^\vee$-conjugacy class of the orbit of $(\varphi_\cus,\varrho_\cus)\in \Phi_\enh^\cusp(M)$ under the action of 
$\fX_\nr({}^LM)$. Let $\fB^\vee(G)$ be the set of such $\fs^\vee$.
 
In \cite{AMS1}, the first author, with Moussaoui and Solleveld, constructed a partition--\`a la Bernstein--of the set $\Phi_\enh(G)$ of $G$-relevant enhanced Langlands parameters:
\begin{equation} \label{eqn:partitionAMS}
\Phi_\enh(G) =\bigsqcup_{\fs^\vee\in \fB^\vee(G)}\Phi_\enh^{\fs^\vee}(G),  
\end{equation}
where $\Phi_\enh^{\fs^{\vee}}(G)$ consists of enhanced Langlands parameters for $G$ whose cuspidal support lies in $\fs^{\vee}$. 
Let $M$ be a Levi subgroup of $G$ and let $\fs^\vee_M:=[{}^LM,\varphi_\cus,\varrho_\cus]_{M^\vee}\in \fB^\vee(M)$. Analogous to the group side $W_G^{\fs}$, we denote by $W^{\fs^\vee}_{G^\vee}$ the stabilizer of $\fs^\vee_M$ in $\Nor_{G^\vee}(M^\vee)/M^\vee$, and
by $W^{\fs^\vee,y}_{G^\vee}$ the stabilizer of $y\in\Phi_\enh^{\fs_M^\vee}(M)$ in $W^{\fs^\vee}_{G^\vee}$.
By \cite[Theorem~9.3]{AMS1}, there is a bijection
\begin{equation}
  \xi_{G^\vee}^{\fs^\vee}\colon\Phi_\enh^{\fs^\vee}(G)\longrightarrow (\Phi_\enh^{\fs_M^\vee}(M)/\!/W^{\fs^\vee}_{G^\vee})_{{}^L\natural},
\end{equation}
where the right-hand side $(\Phi_\enh^{\fs_M^\vee}(M)/\!/W^{\fs^\vee}_{G^\vee})_{{}^L\natural}$ is a twisted extended quotient with respect to a collection $({}^L\natural_y)_y$ of $2$-cocycles 
\begin{equation}
{}^L\natural_y\colon W^{\fs^\vee,y}_{G^\vee}\times W^{\fs^\vee,y}_{G^\vee}\to \C^\times.
\end{equation}

\subsection{Main Results}
Axiomatic setup: we suppose the existence of a map
\begin{equation}\label{introducing-LsM}
\begin{matrix}\fL^{\fs_M}\colon&\Irr^{\fs_M}(M)&\xlongrightarrow&\Phi_\enh^\cusp(M)\cr
    &\sigma&\mapsto &(\varphi_\sigma,\varrho_\sigma)
    \end{matrix}
\end{equation}
such that the following properties are satisfied for any $\sigma\in\Irr^{\fs_M}(M)$:
\begin{enumerate}
    \item[(1)] For any $\chi\in\fX_\nr(M)$, we have
    \[(\varphi_{\chi\otimes\sigma},\varrho_{\chi\otimes\sigma})=\chi^\vee\cdot(\varphi_\sigma,\varrho_\sigma),\]
    where $\chi\mapsto\chi^\vee$ is the canonical isomorphism $\fX_\nr(M)\overset{\sim}{\to}\fX_\nr({}^LM)$.
    \item[(2)]
    For any $w\in W(M)$, we have
 \[{}^{w^\vee}(\varphi_\sigma,\varrho_\sigma)\simeq (\varphi_{{}^w\sigma},\varrho_{{}^w\sigma}),\]
 where $w\mapsto w^\vee$ is the canonical isomorphism  $W(M)\overset{\sim}{\to}W(M^\vee)$.
    \end{enumerate}

We suppose that the collections of $2$-cocycles $\natural$ and ${}^L\natural$ satisfy the following
\begin{equation} \label{eqn:intro matching cocycles}
    {}^L\natural_{\chi^\vee}=\natural_{\chi}\quad\text{for any $\sigma\in\fs$ and any $\chi\in \fX_\nr(M)/\fX_\nr(M,\sigma)$}.
\end{equation}

We establish the following result.

\begin{thm} \label{thmIntro:matching-extended-qts} {\rm (Theorem~\ref{thm:matching-extended-qts})}
\begin{enumerate}
\item 
There is a natural isomorphism
\begin{equation}\fe\colon\Irr^{\fs_M}(M)/\!/W_G^\fs\overset{\sim}{\longrightarrow}\Phi_\enh^{\fs_M^\vee}(M)/\!/W^{\fs^\vee}_{G^\vee}.
\end{equation}
\item The map 
\begin{equation} \label{eqn:LLC}
\fL:=(\xi_{G^\vee}^{\fs^\vee})^{-1}\circ\fe\circ\xi_G^\fs\colon \Irr^\fs(G)\longrightarrow \Phi_\enh^{\fs^\vee}(G)
\end{equation}
is a bijection.
\end{enumerate}
\end{thm}

We suppose in the rest of this introduction that the group $\bG$ splits over a tamely ramified extension of $F$ and that the residual characteristic $p$ of $F$ does not divide the order of the Weyl group of $G$. Then there exists a compact mod center subgroup $\widetilde{K}_M$ of $M$ and an irreducible representation $\rho^d_M$ of it such that $\sigma=\ind_{\widetilde{K}_M}^{M}\rho^d_M$.

Let $\mcH^{\fs}(G)$ denote the endomorphism algebra of the Bernstein progenerator of $\fs$ (see (\ref{eqn:Hs})) and let $\mcH(G,\rho_{\mcD})$ be the intertwining algebra of an $\fs$-type $(K_\mcD,\rho_\mcD)$. We prove in Proposition~\ref{prop:comparisons} that the  algebras $\mcH^{\fs}(G)$ and  $\mcH(G,\rho_\mcD)$ are isomorphic. 

From now on, we suppose that $\sigma$ is \textit{regular} in the sense of \cite{Kal-reg}, which allows us to attach a supercuspidal Langlands parameter $\varphi_\sigma\colon W_F\to {}^LM$ to $\sigma$. Applying Theorem~\ref{thmIntro:matching-extended-qts} to the map $\fL^{\fs_M}\colon\sigma\mapsto (\varphi_\sigma,1)$ as in \eqref{introducing-LsM}, we obtain the following result:

\begin{prop} \label{prop:singleton}
When the $L$-packet of $\sigma$ is a singleton, the properties (1) and (2) are always satisfied.
\end{prop}

On the other hand, the construction of $\widetilde{K}_M$ involves notably a depth zero supercuspidal irreducible representation $\sigma^0$ of a Levi subgroup $M^0$ of a twisted Levi subgroup $G^0$ of $G$. We denote by $\fs^0=[M^0,\sigma^0]_{G^0}$ the inertial class of $\sigma^0$.

Suppose that $p$ is good for $G$ (in the sense of \cite{Carter}) and does not divide the order of the fundamental group of $\bG_\der$, and that the representation $\sigma$ is regular. 
\begin{thm} \label{thm:first} {\rm (Theorem~\ref{thm:xiGG0})} 
There is a bijection
\begin{equation}
\varrho^\fs_{\fs^0}\colon\Irr^\fs(G)\xlongrightarrow{}\Irr^{\fs^0}(G^0),
\end{equation}
which induces a bijection
\begin{equation}
\Irr(\mcH^\fs(G))\xlongrightarrow{}\Irr(\mcH^{\fs^0}(G^0))
\end{equation}
between the sets of equivalence classes of simple modules for the algebras $\mcH^\fs(G)$ and $\mcH^{\fs^0}(G^0)$.
\end{thm}

\smallskip

Theorem~\ref{thm:first} proves the validity of \cite[Conjecture~1.1]{Ad-Mi},  under the above assumption on $p$, for all regular supercuspidal representations of $M$. 
The bijection $\varrho^\fs_{\fs^0}$ is defined as
\begin{equation}
    \varrho^\fs_{\fs^0}:=(\xi_{G^0}^{\fs^0})^{-1}\circ\fl_\sigma\circ\xi_G^\fs,
\end{equation}
where $(\Irr^{{\fs_{M^0}}}(M^0)/\!/W_{G^0}^{\fs^0})_{\natural^0}$ is the twisted extended quotient with respect to a certain collection $\natural^0$ of $2$-cocycles, the definition of which is recalled in (\ref{eqn:cocycleG0}), and
\begin{equation}
\fl_\sigma\colon(\Irr^{\fs_M}(M)/\!/W_G^\fs)_\natural\longrightarrow (\Irr^{{\fs_{M^0}}}(M^0)/\!/W_{G^0}^{\fs^0})_{\natural^0}
\end{equation}
is the isomorphism constructed in \cite{Ad-Mi}.

\smallskip

In Section~\ref{sec:G2}, we study in greater detail the case when $G$ is the exceptional group of type $\rG_2$. Recall that for split $p$-adic groups, the principal series case,~i.e. $M=T$, is due to \cite{Roche-Hecke-algebra}, therefore it suffices to consider the cases where $M\simeq\GL_2(F)$ is a maximal Levi subgroup. The $\rG_2(F)$-covers of the supercuspidal types in $M$ were computed explicitly in \cite{Blondel} when $M$ corresponds to the long simple root of $\rG_2$, and in \cite{Deseine} when $M$ corresponds to the short simple root of $\rG_2$, but the intertwining algebras of these types were still unknown. We compute these intertwining algebras later in \S\ref{subsec:IA}, and in particular, by computing their parameters explicitly, we show that they satisfy a conjecture of Lusztig's in \cite[1.(a)]{Lusztig-Open-problems}.

\bigskip
\noindent

\medskip

\textbf{Acknowledgements.} 
The authors would like to thank Maarten Solleveld for valuable comments on a previous version of the manuscript. Y.X.~was supported by the National Science Foundation under Award No.~2202677 at MIT.

\bigskip
\noindent
\subsection{Notations and Definitions}
Let $F$ be a  non-archimedean local field. Let $\fo_F$ denote the ring of integers of $F$, $\fp_F$ the maximal ideal in $\fo_F$ and $k_F:=\fo_F/\fp_F$ the residue field of $F$. We assume that $k_F$ is finite and denote by  $q_F$ its cardinality. Let $\val_F\colon F\to \Z\cup\{\infty\}$ be a valuation of $F$ and let $\nu_F$ the character of $F^\times$ defined by $\nu_F(a):=q_F^{-\val_F(a)}$ for any $a\in F^\times$.

We fix a separable closure $F_{\sep}$ of $F$. Denote by $W_F\subset\Gal(F_{\sep}/F)$ the absolute Weil group of $F$ and $I_F$ its inertia subgroup.
We denote by $F_{\nr}$ the maximal unramified extension of $F$ inside $F_{\sep}$ and by $\Frob_F$ the element of $\Gal(F_{\nr}/F)$ that induces the automorphism $a\mapsto a^q$ on the residue field $\overline k_F$ of $F_{\nr}$. Then $W_F=I_F\rtimes\langle\Frob_F\rangle$. 
Let $I_F^+$ denote the wild inertia group of $F$ (i.e.~the maximal pro-$p$ open normal subgroup of  $I_F$). We have $I_F^{\tame}=\Gal(F_{\sep}/F_{\tame})\simeq I_F/I_F^+$, where $F_{\tame}$ is the tame closure of $F$ in $F_{\sep}$. The group $I_F^{\tame}$ is pro-cyclic and we denote by $\zeta_F$ a generator of it. Let $W_F':=W_F\times\SL_2(\C)$ be the Weil-Deligne group of $F$.

Let $\bG$ be a connected reductive algebraic group defined over $F$, and $G:=\bG(F)$ its $F$-rational points. 
We denote by $\bG_\der$ the derived group of $\bG$. Let $\bG_{\scn}$ (resp.~$\bG_{\ad}$) be the simply connected cover (resp.~adjoint quotient) of $\bG_\der$. Let $\rZ_\bG$ be the center of $\bG$, and $\bA_\bG$ the maximal $F$-split torus contained in $\rZ_\bG$.

\smallskip
Fix a maximal torus $\bT$ of $\bG$, and let $(X,R,Y,R^\vee)$ denote the root datum of $\bG$ with respect to $\bT$. Thus $X=X^*(\bT)$ is the character group of $\bT$, and $R\subset X$ is the set of weights of $\bT$ on the Lie algebra $\fgg$ of $\bG$.
Fix $\Delta\subset R$ a system of simple roots.
When $R$ is irreducible, the root with maximal height (with respect to $\Delta$) will be
denoted $\widetilde\alpha$. Write
$\widetilde\alpha=\sum_{\gamma\in\Delta} c_\gamma\gamma$
for positive integers $c_\gamma$. A prime number $p$ is said to be \textit{good for $\bG$} if it does not divide any $c_\gamma$. We may simply list the \textit{bad}, i.e.
not good, primes: $p=2$ is bad unless $R$ is of type $\rA$, $p=3$ is bad if $R$ is of type $\rG_2$, $\rF_4$, $\rE_n$, and $p=5$ is bad if $R$ is of type $\rE_8$.
The prime $p$ is good for a general $R$ just in case it is good for each irreducible
component of $R$.

\smallskip

Suppose that $H$ is a group, $H_1$ a subgroup of $H$ and $h$ an element of $H$. We set ${}^hH_1:=hH_1h^{-1}$. If $\pi$ is a representation of $H_1$, we denote by ${}^h\pi$ the representation $h_1\mapsto \pi(h^{-1}h_1h)$ of ${}^hH_1$.  
We denote by $\Irr(H)$ the set of of equivalence classes of irreducible representations of $H$.

The category of right modules over an algebra $\mcA$ is denoted $\mcA\!-\!\Mod$. We write $\Irr(\mcA)$ for the set of equivalence classes of simple modules of $\mcA$.

\subsubsection{Twisted extended quotients} \label{twisted-extended-quotient}
Let $\Gamma$ be a group acting on a topological space $X$ and let $\Gamma_x$ denote the stabilizer in $\Gamma$ of $x\in X$. Let $\natural=(\natural_x)_{x\in X}$ be a collection of $2$-cocycles 
\[\natural_x \colon  \Gamma_x \times \Gamma_x \to \C^\times, 
\]
such that $\natural_{\gamma x}$ and $\gamma_*\natural_x$ define the same class in $H^2(\Gamma_{\gamma x},\C^\times)$, where $\gamma_* \colon  \Gamma_x \to \Gamma_{\gamma x}$
sends $\alpha$ to $\gamma \alpha \gamma^{-1}$. 
Let $\C[\Gamma_x,\natural_x]$ be the group algebra of $\Gamma_x$ twisted by $\natural_x$.
We set
\[\widetilde X_\natural := \left\{(x,\tau) \,:\, \text{$x\in X$, $\tau\in \Irr\,\C[\Gamma_x, \natural_x]$} \right\},
\]
and topologize $\widetilde X_\natural$ by decreeing that a subset of $\widetilde X_\natural$ open if its projection to the first coordinate is open in $X$.

We  require, for every $(\gamma,x) \in \Gamma \times X$, an algebra isomorphism
\[
\phi_{\gamma,x} \colon  \C[\Gamma_x,\natural_x ]  \to \C[\Gamma_{\gamma x},\natural_{\gamma x}]
\]
satisfying the conditions
\begin{itemize}
\item[(a)] if $\gamma x=x$, then $\phi_{\gamma,x}$ is conjugation by an element of
$\C[\Gamma_x,\natural_x]^\times$;
\item[(b)] $\phi_{\gamma',\gamma x} \circ \phi_{\gamma,x} = 
\phi_{\gamma' \gamma,x}$ for all $\gamma',\gamma \in \Gamma$ and $x \in X$.
\end{itemize}
Define a $\Gamma$-action on $\widetilde X_\natural$ by
$\gamma \cdot (x,\tau) := (\gamma x, \tau \circ \phi_{\gamma,x}^{-1})$.
The \textit{spectral twisted extended quotient} of $X$ by $\Gamma$ with respect to $\natural$ is defined to be
\begin{equation} \label{eqn:teq}
(X/\!/ \Gamma)_\natural : = \widetilde{X}_\natural/\Gamma.
\end{equation}
In the case when the $2$-cocycles $\natural_x$ are trivial, we
write simply $X/\!/ \Gamma$ for $(X/\!/ \Gamma)_\natural$ and refer to it as the  \textit{spectral extended quotient} of $X$ by $\Gamma$. 
\addtocontents{toc}{\protect\setcounter{tocdepth}{2}}
\section{Hecke algebras and Bernstein Center} \label{sec:first}
\subsection{General framework} \label{subsec:general}
\begin{numberedparagraph}\label{defining-M1}
Let $\fR(G)$ denote the category of all smooth complex representations of $G$.
It is an abelian category admitting arbitrary coproducts. Let $M$ be a Levi subgroup of a parabolic subgroup $P$ of $G$. We denote by $M_1$ the subgroup of $M$ generated by all its compact subgroups. Recall that a character of $M$ is said to be unramified if it is trivial on $M_1$, and let $\fX_\nr(M)$ be the group of unramified characters of $M$. Let $\sigma$ be an irreducible supercuspidal smooth representation of $M$. We write $\fs:=[M,\sigma]_G$ for the $G$-conjugacy class of the pair $(M,\fX_\nr(M)\cdot\sigma)$, it is called a \textit{Bernstein inertial class}.  Let $\fB(G)$ denote the set of Bernstein inertial classes $\fs$. We set $\fs_M:=[M,\sigma]_M$.

We denote by $\fR^\fs(G)$ the full subcategory of $\fR(G)$ whose objects are  the representations $(\pi,V)$ such that every $G$-subquotient of $\pi$ is equivalent to a subquotient of a parabolically induced representation $i_P^G(\sigma')$, where $i_P^G$ is the functor of normalized parabolic induction  and $\sigma'\in\fX_\nr(M)\cdot\sigma$. We write $\Irr^\fs(G)$ for the class of irreducible objects in $\fR^\fs(G)$, i.e.~representations whose supercuspidal support lies in $\fs$.
\end{numberedparagraph}

\begin{numberedparagraph}
The categories $\fR^\fs(G)$ are indecomposable and split the full smooth category $\fR(G)$ in a direct product (see \cite[Proposition~2.10]{Ber84}):
\[\fR(G)=\prod_{\fs\in\fB(G)}\fR^\fs(G).\]
If $\Pi^\fs$ is a progenerator of $\fR^\fs(G)$, then the functor $V\mapsto \Hom_G(\Pi^\fs,V)$ is an equivalence from $\fR^\fs(G)$ to the algebra $\End_G(\Pi^\fs)$ (see for instance \cite[\S~1.1]{Roche-Bernstein-decomposition}).

Let $\fs=[M,\sigma]_G\in\fB(G)$ and  let $V$ be the underlying vector space for the supercuspidal representation $\sigma$ of $M$ and $\sigma_1$  an irreducible component of the restriction of $\sigma$ to $M_1$. We denote by $\ind_{M_1}^M$ the functor of compact induction. As noticed in \cite[\S~1.2]{Roche-Bernstein-decomposition}, the isomorphism class of 
\begin{equation} \label{eqn:cusp prog}
    \Pi_M^{\fs_M}:=\ind_{M_1}^M (\sigma_1)
\end{equation}
is independent of the choice of $\sigma_1$. 
It was shown by Bernstein that 
\begin{equation} \label{eqn:prog}
    \Pi_G^\fs:=i_P^G(\Pi_G^{\fs_M})
\end{equation}
is a progenerator of $\fR^\fs(G)$ (see \cite[\S 1.6]{Roche-Bernstein-decomposition}). We write
\begin{equation} \label{eqn:Hs}
    \mcH^\fs(G):=\End_G(\Pi_G^\fs).
\end{equation}
Hence we have an equivalence of categories of right modules
\begin{equation} \label{eqn:equiv3}
 \fR^\fs(G)\,\simeq\,
\mcH^\fs(G)-\Mod.
\end{equation}

Let $B:=\C[M/M_1]$ and $V_B:=V\otimes_\C B$. Then $i_P^G(V_B)$ is also a progenerator of $\fR^\fs(G)$, and  we have an equivalence of categories of right modules given by
\begin{equation} \label{eqn:equivE}
\begin{matrix}
\mcE\colon& \fR^\fs(G)&\longrightarrow&\End_G(i_P^G(V_B))\!-\!\Mod\cr
&\mcV&\mapsto &\Hom_G(i_P^G(V_B),\mcV)
\end{matrix}.
\end{equation}
\end{numberedparagraph}

\begin{numberedparagraph}
Consider
\begin{equation} \label{eqn:XnrMsigma}
\fX_{\nr}(M,\sigma):=\{\chi\in\fX_{\nr}(M)\,:\,\chi\otimes\sigma\simeq\sigma\},
\end{equation}
which is a finite subgroup of $\fX_{\nr}(M)$.

\begin{remark} \label{rem:size}
In the case where $M=\GL_n(F)$ with $n$ a positive integer, there is a simple type $(J,\lambda)$ in the sense of \cite[(5.5.10)]{BKbook} such that the restriction of the supercuspidal representation $\sigma$ to $J$ contains $\lambda$. The order of $\fX_{\nr}(M,\sigma)$ is $n/e(L|F)$, where $e(L|F)$ is the ramification index of the extension $L/F$ involved in the definition of $(J,\lambda)$ (see \cite[(6.0.1) and (6.2.5)]{BKbook}.
\end{remark}

We denote by $\Oo$ the orbit of $\sigma$ under the action of $\fX_\nr(M)$. The map $\chi\mapsto \chi\otimes\sigma$ defines a bijection  
\begin{equation} \label{eqn:IrrsM}
\fX_\nr(M)/\fX_\nr(M,\sigma)\xrightarrow{\sim}\Oo=\left\{\chi\otimes\sigma\,:\,\chi \in \fX_\nr(M)\right\}=\Irr^{\fs_M}(M).
\end{equation}

\smallskip

We set $W(M):=\Nor_{G}(M)/M$ and define  
\begin{equation} \label{eqn:Ws}
    W^\fs_G:=W(M,\Oo):=\left\{n\in \Nor_{G}(M)\,:\,{}^n\Oo\simeq \Oo\right\}/M.
\end{equation}

Recall that $\bA_M$ is the maximal split torus contained in the center of $\bM$. We denote by $\Sigma(A_M)\subset X^*(A_M)$ the set of nonzero weights occurring in the adjoint representation of $A_M$ on the Lie algebra of $G$, and by $\Sigma_\red(A_M)$ be the set of indivisible elements therein. (Recall that a root $\gamma$ in a root system $\Sigma$ is called \textit{indivisible} if $\frac{1}{2}\gamma\notin\Sigma$.)

For every $\gamma\in \Sigma_\red(A_M)$, let $M_\gamma\supset M$ 
denote the centralizer of $\ker \gamma$ in $G$ (it is a Levi subgroup of $G$ whose semisimple rank is one larger than that of $M$). 
Let $\mu^G$ be the \textit{Harish-Chandra $\mu$-function} for $G$ (see \cite[\S 1]{Silberger-79} or \cite[\S V.2]{Wal}). The restriction of $\mu^G$ to $\Oo$ is a rational $W(M,\Oo)$-invariant function on $\Oo$ \cite[Lemma V.2.1]{Wal}.
By \cite[Proposition1.3]{Heiermann-intertwining-operators-Hecke-algebras}, the set
\begin{equation} \label{eqn:Root System}
    \Sigma_{\Oo,\mu}:=
\left\{\gamma\in \Sigma_{\red}(A_M)\,:\,
\text{$\mu^{M_\gamma}$ has a zero on $\Oo$}\right\}
\end{equation}
is a root system. 
Let $W_{\Oo}$ denote the Weyl group of $\Sigma_{\Oo,\mu}$. 

Let $P=MN$ be a parabolic subgroup of $G$ with Levi factor $M$. Denote by $\Sigma(P)$ the subset of $\Sigma(A_M)$ of roots which act on the Lie algebra of $N$. Let $\Sigma_{\Oo,\mu}(P):=\Sigma_{\Oo,\mu}\cap\Sigma(P)$. By \cite[1.12]{Heiermann-intertwining-operators-Hecke-algebras}, the group $W(M,\Oo)$ decomposes as
\begin{equation} \label{eqn:WMO}
W(M,\Oo)=W_\Oo\rtimes R(\Oo),
\end{equation}
where 
\begin{equation} \label{eqn:RO}
R(\Oo):=\left\{w\in W(M,\Oo)\,:\, w(\Sigma_{\Oo,\mu}(P))=\Sigma_{\Oo,\mu}(P)\right\}.
\end{equation}
The action of every element $w$ of $W_G^\fs$ can be lifted to a transformation $\tilde w$ of $\fX_\nr(M)$. Let $W(M,\sigma,\fX_\nr(M))$ be the group of permutations of $\fX_\nr(M)$ generated by $\fX_\nr(M,\sigma)$ and the $\tilde w$'s. We have
\begin{equation} \label{eqn:WMX}
W(M,\sigma,\fX_\nr(M))/\fX_\nr(M,\sigma)\simeq W_G^\fs.
\end{equation}
Let $K(B):=\C(\fX_\nr(M))$ denote the quotient field of $B:=\C[\fX_\nr(M)]$. Let 
\[\C[W(M,\sigma,\fX_\nr(M)),\kappa]\] 
be the twisted group algebra of $W(M,\sigma,\fX_\nr(M))$ with basis elements $t_w$ that multiply as $t_wt_{w'} =\kappa(w,w')t_{ww'}$. 
By \cite[Corollary~5.8]{Solleveld-endomorphism-algebra}, there is a $2$-cocycle 
\begin{equation}
   \kappa\colon  W(M,\sigma,\fX_\nr(M))\times W(M,\sigma,\fX_\nr(M))\to \C^\times,
\end{equation}
such that we have an algebra isomorphism
\begin{equation}
K(B)\otimes_B \End_G(i_P^G(V_B)\simeq
K(B)\rtimes\C[W(M,\sigma,\fX_\nr(M)),\kappa].
\end{equation}
Here the symbol $\rtimes$ denotes the crossed product: as a vector space, it just means the tensor product, with multiplication rules determined by the action of $W(M,\sigma,\fX_\nr(M))$ on $K(B)$. Note that the cocycle 
$\kappa$ is trivial on $W_\Oo$.
\end{numberedparagraph}

\begin{remark} \label{rem:kappa trivial}
If $R(\Oo)$ has order at most $2$, the intertwining operators can be normalized such that the cocycle $\kappa$ is 
trivial (see \cite[Proposition~5.2 \& above Lemma~5.7]{Solleveld-endomorphism-algebra}).
This is indeed the case for $G=\rG_2(F)$.
\end{remark}

For any $\chi\in\fX_\nr(M)$, let $W_G^{\fs,\chi\otimes\sigma}$ denote the stabilizer of $\chi\otimes\sigma$ in $W^\fs_G$. 
Let $\natural_\chi$ be the $2$-cocycle denoted $\natural_{\sigma'}$ in \cite[(9.13)]{Solleveld-endomorphism-algebra}. 
Let $(\Irr^{\fs_M}(M)/\!/W_G^\fs)_\natural$ denote
the twisted extended quotient (as in $\mathsection$\ref{twisted-extended-quotient}) with respect to the collection $\natural$ of the $2$-cocycles $\natural_\chi$.

\begin{prop} \label{prop:bijectionsS}
There is a bijection 
\begin{equation} \label{eqn:xiG}
 \xi_G^\fs\colon   \Irr^\fs(G) \longrightarrow (\Irr^{\fs_M}(M)/\!/W_G^\fs)_\natural.
\end{equation}
\end{prop}
\begin{proof}
By \cite[Theorem~9.7]{Solleveld-endomorphism-algebra}, there are bijections
\[\Irr^\fs(G)\overset{\mcE}{\longleftrightarrow} \Irr(\End_G(i_P^G(V_B))\overset{\zeta}{\longleftrightarrow}\Irr(\C[\fX_\nr(M)]\rtimes\C[W(M,\sigma,\fX_\nr(M)),\kappa],\]
where $\mcE$ is induced by the equivalence of categories defined in (\ref{eqn:equivE}).
On the other hand, by \cite[Lemma~9.8]{Solleveld-endomorphism-algebra}, $\Irr(\C[\fX_\nr(M)]\rtimes\C[W(M,\sigma,\fX_\nr(M)),\kappa]$ is canonically isomorphic to  $(\Irr^{\fs_M}(M)/\!/W_G^\fs)_\natural$, where ${\fs_M}:=[M,\sigma]_M$. 
\end{proof}

\begin{Coro} \label{cor:two_algebras}
Let $\fs=[M,\sigma]_G\in\fB(G)$. There is a bijection
\begin{equation}
\Irr(\mcH^\fs(G))\xlongrightarrow{1-1}(\Irr^{\fs_M}(M)/\!/W_G^\fs)_\natural.
\end{equation}
\end{Coro}
\begin{proof}
The result follows from the proof of Proposition~\ref{prop:bijectionsS} by using (\ref{eqn:equiv3}).
\end{proof}

\begin{remark} \label{rem:two_algebras}
As observed in \cite[(10.12)]{Solleveld-endomorphism-algebra}, if the restriction of $\sigma$ to $M_1$ is multiplicity free, we have
\begin{equation} \label{eqn:Morita}
\Pi_G^\fs=\left(i_P^G(V_B)\right)^{\fX_\nr(M,\sigma)}
\quad\text{and}\quad
\End_G(i_P^G(V_B))\simeq \mcH^\fs(G) \otimes_\C \Mat_{[M:M_\sigma]}(\C),
\end{equation}
where $\Mat_{[M:M_\sigma]}(\C)$ is the algebra of square matrices of size $[M:M_\sigma]$ (the index of $M_\sigma$ in $M$) with entries in $\C$.
Note that if $\sigma$ is generic, then its restriction to $M_1$ is multiplicity free (see \cite[Remark 1.6.1.3]{Roche-Bernstein-Center}). In particular, if $\sigma$ is a supercuspidal irreducible representation of a proper Levi subgroup $M$ of $\rG_2$, since $M$ is isomorphic to either $F^\times\times F^\times$ or $\GL_2(F)$, the representation $\sigma$ is generic, and hence its restriction to $M_1$ is multiplicity free.
\end{remark}

\subsubsection{Theory of types}
We fix a Haar measure on $G$. Let $\mcH(G)$ be the space of locally constant, compactly supported functions $f \colon G \to\C$ and view $\mcH(G)$ as a $\C$-algebra via convolution relative to the Haar measure. The algebra $\mcH(G)$ is called the \textit{Hecke algebra of $G$}. 

Let $(\rho,V_\rho)$ be a smooth representation of a compact open subgroup $K$ of $G$, and let $(\tilde\rho,V_{\tilde\rho})$ denote its contragredient. We define $\mcH(G,\rho)$ to be the space of compactly supported functions $f\colon G \to \End_G(V_{\tilde\rho})$ such that
\begin{equation} \label{eqn:Hecke}
f(kgk')=\tilde\rho(k)f(g)\tilde\rho(k'),\quad \text{where $k,k'\in K$ and $g\in G$.}\end{equation}
The convolution product gives $\mcH(G,\rho)$ the structure of a unitary associative $\C$-algebra. The algebra $\mcH(G,\rho)$ is called the \textit{$\tilde\rho$-spherical Hecke algebra} or the \textit{intertwining algebra of $(K,\rho)$}.

Let $e_\rho\in\mcH(G)$ be the function defined by
\begin{equation} \label{eqn:idempotent}
e_\rho(g):=\begin{cases} 
\frac{\dim\rho}{\meas(K)}\trace(\rho(g^{-1}))&\text{if $g\in K$},\cr0&\text{if $g\in G$, $g\notin K$.}
\end{cases}
\end{equation}
Then $e_\rho$ is idempotent, and $e_\rho\star\mcH(G)\star e_\rho$ is a sub-algebra of $\mcH(G)$ with unit $e_\rho$. 
By \cite[(2.12)]{Bushnell-Kutzko}, there is a canonical isomorphism 
\begin{equation}
\mcH(G,\rho)\otimes_\C\End_\C(V_\rho)\,\to\, e_\rho\star\mcH(G)\star e_\rho.
\end{equation}
The algebras $\mcH(G,\rho)$ and $e_\rho\star\mcH(G)\star e_\rho$ are therefore canonically Morita equivalent. Hence, we get an equivalence of categories: 
\begin{equation} \label{eqn:equiv1}
\mcH(G,\rho)-\Mod\,\simeq
\, e_\rho\star\mcH(G)\star e_\rho-\Mod .
\end{equation}
Let $\fR_\rho(G)$ be the full subcategory of $\fR(G)$ whose objects are those $V$ satisfying $V=\mcH(G)\star e_\rho\star V$, i.e.~$\fR_\rho(G)$ is generated over $G$ by the subspace $e_\rho\star V$.

\begin{Defn} \label{def:type}
(1) The pair $(K,\rho)$ is called an \textit{$\fs$-type for $G$} if the category $\fR_\rho(G)$ is closed under subquotients. \\
(2) A \textit{supercuspidal type for $G$} is an $\fs$-type where $\fs=[G,\sigma]_G$.
\end{Defn}

If $(K,\rho)$ is an $\fs$-type for $G$, then $\fR_\rho(G)=\fR^\fs(G)$ by \cite[(4.1)--(4.2)]{Bushnell-Kutzko}, where $\fR^\fs(G)$ is equivalent to the category of modules for $\mcH(G,\rho)$ by \cite[Theorem~3.5]{Bushnell-Kutzko}:
\begin{equation} \label{eqn:equiv2}
 \fR^\fs(G)\,\simeq\,
\mcH(G,\rho)\!-\,\Mod.
\end{equation}
Combining (\ref{eqn:equiv2}) and (\ref{eqn:equiv3}), we obtain an equivalence
\begin{equation}
  \mcH^\fs(G)\!-\!\Mod\,  \simeq  \,\mcH(G,\rho)\,-\,\Mod.
\end{equation}
Let $(K_M,\rho_M)$ be an $\fs_M$-type for $\fs_M\in\fB(M)$. If the pair $( K,\rho)$ is  a $G$-cover of $(K_M,\rho_M)$ as defined in \cite[Definition 8.1]{Bushnell-Kutzko}, then $K$ decomposes with respect to $M$ in the sense of \cite[Definition~6.1]{Bushnell-Kutzko} (in particular, $K_M= K\cap  M$  and $\rho_M = \rho|_{K_M }$) and the equivalence of categories (\ref{eqn:equiv2}) commutes with parabolic induction and parabolic restriction in the appropriate sense (see \cite[Corollary~8.4]{Bushnell-Kutzko}). 

\begin{prop} \label{prop:iso_progenerators}
Let $(K_M,\rho_M)$ be an $\fs_M$-type for $\fs_M\in\fB(M)$, such that $\Pi_M^{\fs_M}\simeq\cInd_{K_M}^M(\rho_M,V_{\rho_M})$. Let $(K,\rho)$ be a $G$-cover of $(K_M,\rho_M)$. 
Then 
\begin{equation}
\Pi_G^\fs\simeq\cInd_{K}^G(\rho,V_\rho).
\end{equation}
As a consequence, we have
\begin{equation} \label{eqn: iso HA}
  \mcH^\fs(G):=\End_G(\Pi_G^\fs) \simeq \mcH(G,\rho).
\end{equation}
\end{prop}
\begin{proof}
See \cite[Lemma~B.3]{Bakic-Savin}. 
\end{proof}

\begin{numberedparagraph}
In this section, in order to be able to apply the constructions of \cite{Yu} and \cite{Kim-Yu}, we assume that $\bG$ splits over a tamely ramified extension of $F$, and that $p$ does not divide the order of the Weyl group of $G$. 
By a \textit{Levi subgroup of $\bG$}, we mean an $F$-subgroup of $\bG$ which is a Levi factor of a parabolic $F$-subgroup of $\bG$. Let $L/F$ be a finite extension. By a \textit{twisted $L$-Levi subgroup of $\bG$}, we mean an $F$-subgroup $\bG'$ of $\bG$ such that $\bG'\otimes_F L$ is a Levi subgroup of $\bG\otimes_F L$. If $L/F$ is tamely ramified, then $\bG'$ is called a \textit{tamely ramified twisted Levi subgroup of $\bG$}. A \textit{tamely ramified twisted Levi sequence in $\bG$} is a finite sequence $\vec \bG=(\bG^0,\bG^1,\cdots,\bG^d)$ of twisted $E$-Levi subgroups of $\bG$, with $E/F$ tamely ramified (see \cite[p~586]{Yu}). 

Let $\mcB(\bG,F)$ denote the (enlarged) building of $G$:
\begin{equation}
\mcB(\bG,F)=\mcB(\bG/\rZ_{\bG},F)\times X_*(\rZ_{G})\otimes_\Z\R,
\end{equation}
where $X_*(\rZ_G)$ is the set of $F$-algebraic cocharacters of $\rZ_G$.
Recall that when $\bG'$ is a tamely ramified twisted Levi subgroup of $\bG$, there is a family of natural embeddings of $\mcB(\bG',F)$ into $\mcB(\bG,F)$.

For $x$ a point in $\mcB(\bG,F)$, let $G_{x,0}$ denote the associated parahoric subgroup, and let $G_{x,0^+}$ denote the pro-$p$ unipotent radical of $G_{x,0}$. In general, for $r$ a positive real number, $G_{x,r}$ is the corresponding Moy-Prasad filtration subgroup of $G_{x,0}$. 

\begin{Defn} \cite[\S~7.1]{Kim-Yu} 
A \textit{depth-zero $G$-datum} is a triple 
\begin{equation} \label{depth-zero datum}
((\bG,\bM),(y,\iota),(K_M,\rho_M))
\end{equation}
satisfying the following
\begin{itemize}
\item $\bG$ is a connected reductive group over $F$, and $\bM$ is a Levi subgroup of $\bG$;
\item $y$ is a point in $\mcB(M)$ such that $M_{y,0}$ is a maximal parahoric subgroup of $M$, and $\iota\colon\mcB(M)\hookrightarrow\mcB(G)$ is a $0$-generic embedding relative to $y$ (see \cite[Definition~3.2]{Kim-Yu});
\item $K_M$ is a compact open subgroup of $M$ containing $M_{y,0}$ as a normal subgroup, and $\rho_M$ is an irreducible smooth representation of $K_M$ such that $\rho_M|M_{y,0}$ contains the inflation to $M_{y,0}$ of a cuspidal representation of $M_{y,0}/M_{y,0^+}$.
\end{itemize}
\end{Defn}
Let $\vec \bG=(\bG^0,\bG^1,\cdots,\bG^d)$ be a tamely ramified twisted Levi sequence in $\bG$.
To $\vec \bG$, we associate a sequence of Levi subgroups $\vec \bM=(\bM^0,\cdots,\bM^d)$, where $\bM^i$ is a Levi subgroup of $\bG^i$ given as the centralizer of $A_{\bM^0}$ in $\bG^i$, with $A_{\bM^0}$ the maximal $F$-split torus of the center $\rZ_{\bM^0}$ of $M^0$. 
\begin{Defn}
A \textit{$\bG$-datum} is a $5$-tuple 
\begin{equation} \label{eqn:Gdatum}
\mcD=((\vec \bG,\bM^0),(y,\{\iota\}),\vec r,(K_{M^0},\rho_{M^0}),\vec\phi)
\end{equation}
satisfying the following:

\begin{itemize}
\item[\bf D1.] $\vec \bG=(\bG^0,\bG^1,\cdots, \bG^d)$ is a tamely ramified twisted Levi sequence in $\bG$, and $\bM^0$ a 
Levi subgroup of $\bG^0$. Let $\vec \bM$ be associated to $\vec \bG$ as above;
\item[\bf D2.] $y$ is a point in $\mcB(M^0)$, and $\{\iota\}$ is a commutative diagram of $\vec s$-generic embeddings of buildings relative to $y$ in the sense of \cite[Definition~3.5]{Kim-Yu}, where $\vec s=(0,r_0/2,\cdots,r_{d-1}/2)$;
\item[\bf D3.] $\vec r=(r_0,r_1,\cdots,r_d)$ is a sequence of real numbers satisfying $0<r_0<r_1<\cdots<r_{d-1}\le r_d$ if $d>0$, and $0\le r_0$ if $d=0$;
\item[\bf D4.] $(K_{M^0},\rho_{M^0})$ is such that $\mcD^0:=((\bG^0,\bM^0),(y,\iota), (K_{M^0},\rho_{M^0}))$ is a depth zero $G^0$-datum;
\item[\bf D5.] $\vec\phi=(\phi_0,\phi_1,\cdots,\phi_d)$ is a sequence of quasi-characters, where $\phi_i$ is a quasi-character of $G^i$ such that $\phi_i$ is
$G^{i+1}$-generic of depth $r_i$ relative to $x$ for all $x\in\mcB(G^i)$ in the sense of \cite[\S~9]{Yu}.
\end{itemize}
\end{Defn}

\noindent
{\bf The construction.}
For a given $\bG$-datum $\mcD$ as in \eqref{eqn:Gdatum}, we write
\begin{equation} \label{eqn:KD0}
K_{\mcD^0}:=K_{M^0}G^0_{\iota(y),0}.
\end{equation}
We recall that $G^0_{\iota(y),0}$ is the parahoric subgroup of $G^0$ associate to the building point $\iota(y)$.  Let $G^0_{\iota(y),0+}$ be the pro-$p$ unipotent radical of $G^0_{\iota(y),0}$.

By \cite[Proposition~4.3(b)]{Kim-Yu}, we have
\begin{equation} \label{eqn:rqt}
    K_{\mcD^0}/G^0_{\iota(y),0+}\simeq K_{M^0}/M_{y,0+}^0,
\end{equation}
and we define $\rho_{\mcD^0}$ to be the representation of $K_{\mcD^0}$ obtained by composing the isomorphism (\ref{eqn:rqt}) with $\rho_{M^0}$.

\begin{Defn}
A \textit{Kim-Yu type} in the sense of \cite[\S7.4]{Kim-Yu}, which builds on earlier construction in \cite{Yu}, is a pair $(K_\mcD,\rho_\mcD)$ where 
\begin{itemize}
    \item $K_\mcD$ is an open compact subgroup given by
    \begin{equation} \label{eqn:KSigma}
K_\mcD:=K_{\mcD^0} G^1_{\iota(y),s_0}\,\cdots \,G^d_{\iota(y),s_{d-1}}
\end{equation}
\item $\rho_\mcD$ is an irreducible representation of $K_\mcD$.
\end{itemize}
\end{Defn}

To $\vec \bG$, we associate a tamely ramified twisted Levi sequence $\vec \bM=(\bM^0,\ldots, \bM^d)$ of $\bM$, where $\bM^i$ is the centralizer of $A_\bM$ in $\bG^i$. 
Consider 
\begin{equation} \label{eqn:Mdatum}
\mcD_M:=(\vec \bM,y,\vec r,\rho_{M^0},\vec\phi).
\end{equation}
When $K_{M^0}=M_y^0$ the datum  $\mcD_M$ gives a supercuspidal type in $M$ as follows.

Let $K^d_M:=K_\mcD\cap M$. 
Let $\widetilde{K}^d_{M}$ denote the normalizer in $M$ of $K_M^d$. This group $\widetilde{K}^d_{M}$ is a compact modulo center subgroup of $M$. Let $\rho^d_M:=\rho_\mcD|_{K_M^d}$ and consider
\begin{equation} \label{eqn:sigma}
\sigma_{\mcD_M}:=\ind_{\widetilde{K}_M}^{M}\rho^d_M.
\end{equation}
 
\begin{theorem} {\rm \cite{Kim-Yu}}
Suppose that $K_{M^0}=M_y^0$. Then 
\begin{itemize}
\item[\rm (1)]
$(K^d_M,\rho^d_M)$ is a supercuspidal type on $M$ (as in Definition~\ref{def:type}), and $\sigma_{\mcD_M}$ is an irreducible supercuspidal representation of $M$;
\item[\rm (2)]
$(K_\mcD,\rho_\mcD)$ is a $G$-cover of $(K^d_M,\rho^d_M)$, thus it is an $\fs$-type for $\fs=[M,\sigma_{\mcD_M}]_G$.
\end{itemize}
\end{theorem}

\begin{prop} \label{equiv-types}
Let $\mcD$ and $\dot\mcD$ be two $G$-data
\[\mcD=((\vec \bG,\bM^0),(y,{\iota}),\vec r,(K_{M^0},\rho_{M^0}),\vec\phi)\;\;\text{and}\;\; \dot\mcD=((\vec{\dot\bG},\dot\bM^0),(\dot y,\dot{\iota}),\vec{\dot r},(K_{\dot M^0},\rho_{\dot M^0}),\vec{\dot{\phi}}\,)\]  such that $K_{M^0}=M^0_y$ and $K_{\dot M^0}=\dot M^0_{\dot y}$. Let
$\fs:=[M,\sigma_{\mcD_M}]_G$, and $\dot\fs=[M,\sigma_{\dot\mcD_M}]_G$.

Then  we have
$\fs=\dot\fs$ if and only if there exists $g\in G$ such that 
\begin{equation} \label{eqn:condition}
{}^gK_{M^0}=K_{\dot M^0}\quad\text{and}\quad {}^g(\rho_{M^0}\otimes\phi)\simeq \rho_{\dot M^0}\otimes\dot \phi,\end{equation}
where $\phi:=\prod_{i=0}^d(\phi_i|_{M^0})$ and $\dot\phi:=\prod_{i=0}^d(\dot\phi_i|_{M^0})$.
\end{prop}
\begin{proof} It is a reformulation of \cite[Theorem~10.3]{Kim-Yu}. Indeed, when $K_{M^0}=M^0_y$ and $K_{\dot M^0}=\dot M^0_{\dot y, 0}$, we have $\fs=\dot\fs$ if and only if the types $(K_{\mcD},\rho_{\mcD})$ and $(K_{\dot\mcD},\rho_{\dot\mcD})$ are equivalent in the sense of \cite[Definition~10.1]{Kim-Yu}. 
Note that \cite[Theorems~10.2 and 10.3]{Kim-Yu} still hold without assuming the hypothesis $C(\vec\bG)$ of \cite[Remark~2.49 \& above]{Hak-Mur}, since \cite[\S3.5]{Kal-reg} shows that \cite[Theorems~6.6 and~6.7]{Hak-Mur} are valid without assuming $C(\vec\bG)$.
\end{proof}

\begin{remark} \label{remark:HM}
If $G=M$, it follows from \cite[Theorems~6.6 and~6.7]{Hak-Mur} that $\fs_M=\dot\fs_M$ if and only the data $\mcD_M=(\vec \bM,y,\vec r,\rho_{M^0},\vec\phi)$ and $\dot\mcD_M=(\vec \bM,\dot y,\vec {\dot r},\rho_{\dot M^0},\vec{\dot \phi})$ are equivalent in the sense of \cite[Definition~5.3]{Hak-Mur}.
\end{remark}
\begin{numberedparagraph}
For any $\fs=[M,\sigma]_G\in\fB(G)$, consider
\begin{equation} \label{eqn:Ns}
\Nor_G^\fs:=\left\{n\in \Nor_\bG(\bM)(F)\,:\,{}^n\sigma\simeq\chi\otimes\sigma\;\text{ for some $\chi\in\fX_\nr(M)$}\right\}.
\end{equation}

\begin{Coro} We suppose that $p$ does not divide the order of the Weyl group of $G$.
Let $\fs\in\fB(G)$ be an arbitrary Bernstein inertial class.
For every $n\in \Nor_G^\fs$, there exists an $m\in M$ such that
\[{}^{mn}K_{M^0}=K_{M^0}\quad\text{and} \quad {}^{mn}(\rho_{M^0}\otimes\phi)\simeq \rho_{M^0}\otimes\phi,\]
where $\phi:=\prod_{i=0}^d(\phi_i|_{M^0})$.
\end{Coro}
\begin{proof}
By \cite{Fintzen}, we have $\sigma=\sigma_{\mcD_M}$, for some $M$-datum $\mcD_M=(\vec \bM,y,\vec r,\rho_{M^0},\vec\phi)$.
Let $n\in \Nor_G^\fs$. Thus ${}^n\sigma\simeq\chi\otimes\sigma$ for some $\chi\in\fX_\nr(M)$. Let $\dot\mcD_M:=(\vec \bM,y,\vec r,\rho_{M^0}\otimes\chi|_{K_{M^0}},\vec{\phi})$.
By (\ref{eqn:sigma}), we have 
\begin{equation}
\chi\otimes\sigma=(\ind_{\widetilde{K}_M}^{M}\rho^d_M)\otimes\chi\simeq\ind_{\widetilde{K}_M}^{M}(\rho^d_M\otimes\chi|_{\widetilde{K}_M}).
\end{equation}
Since $\chi$ is unramified, we have $\ind_{\widetilde{K}_M}^{M}(\rho^d_M\otimes\chi|_{\widetilde{K}_M})=\sigma_{\dot\mcD_M}$. Therefore $\chi\otimes\sigma\simeq\sigma_{\dot\mcD_M}$. 
Applying Remark~\ref{remark:HM} to the $M$-data
${}^n\mcD_M$ and $\dot\mcD_M$, one can see that 
these data are equivalent. 
Hence there exists an $m\in M$ such that
\[{}^{mn}K_{M^0}=K_{M^0}\quad\text{and} \quad {}^{mn}(\rho_{M^0}\otimes\phi)\simeq \rho_{M^0}\otimes\chi|_{K_{M^0}}\otimes \phi=\rho_{M^0}\otimes\phi,\]
where the last equality holds because $\chi$ is trivial on $K_{M^0}$. Here $\phi:=\prod_{i=0}^d(\phi_i|_{M^0})$.
\end{proof}
\end{numberedparagraph}

\begin{numberedparagraph}
As shown in \cite{Hak-Mur}, it follows from the original construction in \cite{Yu} that $\rho^d_M$ is of the form $\rho_M^d=\rho_{M^0}\otimes\kappa$, where the representation $\kappa=\kappa_G$ depends only on $\vec\phi$. 
Suppose $K_{M^0}=M^0_y$. Let $\widetilde{K}_{M^0}$ denote the normalizer of $K_{M^0}$ in $M^0$. Consider
\begin{equation} \label{eqn:sigma0}
\sigma^0:=\ind_{\widetilde{K}_{M^0}}^{M^0}\rho_{M^0}.
\end{equation}
The representation 
$\sigma^0$ is a depth-zero irreducible supercuspidal representation of $M^0$. 
Let $\fs^0:=[M^0,\sigma^0]_{G^0}$. Let $\Pi_{G^0}^{\fs^0}$ be defined as in (\ref{eqn:prog}), i.e.~it is a progenerator for the category $\fR^{\fs^0}(G^0)$. Set $\mcH^{\fs^0}(G^0):=\End(\Pi_{G^0}^{\fs^0})$.

\begin{prop} \label{prop:comparisons} 
\begin{enumerate}
    \item The algebras $\mcH^{\fs^0}(G^0)$ and  $\mcH(G^0,\rho_{\mcD^0})$ are isomorphic.
    \item The algebras $\mcH^{\fs}(G)$ and  $\mcH(G,\rho_\mcD)$ are isomorphic.
\end{enumerate}
\end{prop}
\begin{proof} We verify the assumptions in Proposition~\ref{prop:iso_progenerators}. Firstly, $(K_{\mcD^0},\rho_{\mcD^0})$ is a $G^0$-cover of $(K_{M^0},\rho_{M^0})$ (see \cite[\S 7.1]{Kim-Yu}) and $(K_\mcD,\rho_\mcD)$ is a $G$-cover of $(K^d_M,\rho^d_M)$ (see \cite[Theorem~7.5]{Kim-Yu}). 

Secondly, since $\sigma^0$ and $\sigma$ are supercuspidal irreducible representations, an element $m^0$ of $M^0$ intertwines $\rho_{M^0}$ if and only if $m^0\in \widetilde{K}_{M^0}$, and an element $m$ of $M$ intertwines $\rho_M$ if and only if $m\in \widetilde{K}_{M}$. 
Then the proof of \cite[Lemma B.4]{Bakic-Savin} applies, and shows that $\Pi^{\fs_M}_M\simeq\cInd_{K_M}^M(\rho_M,V_{\rho_M})$. Then the result follows from Proposition~\ref{prop:iso_progenerators}.
\end{proof}

\begin{prop} \label{prop:trivial case}
If $W_G^\fs=\{1\}$, then there is an algebra isomorphism 
\[\mcH(G,\rho_\mcD)\simeq \mcH(M,\rho_M^d),\]
which preserves support of functions,
and the algebra $\mcH(G,\rho_\mcD)$ is commutative.
\end{prop}
\begin{proof}
Since $W_G^\fs=\{1\}$, we have $\Nor_G(\fs)\subset M$. Then the first assertion follows from \cite[(12.1)]{Bushnell-Kutzko}. On the other hand, the algebra $\mcH(M,\rho_M^d)$ is commutative (see for instance \cite[(5.6)]{Bushnell-Kutzko}).
\end{proof}

\begin{remark} \label{rem: trivial commutative case}
Applying Proposition~\ref{prop:trivial case} to the group $G^0$, we see that if $W_{G^0}^{\fs^0}=\{1\}$, then there is
an algebra isomorphism 
\begin{equation}
\mcH(G^0,\rho_{\mcD^0})\simeq \mcH(M^0,\rho_{M^0}^0)
\end{equation}
that preserves support of functions;
thus the algebra $\mcH(G^0,\rho_{\mcD^0})$ is also commutative.
\end{remark}
\end{numberedparagraph}
\end{numberedparagraph}

\subsection{
Bernstein blocks} \label{subsec:ns}
We assume that $G$ splits over a tamely ramified extension, and that the residual characteristic $p$ of $F$ is odd, good for $G$ and does not divide the order of the fundamental groups of $\bG_\der$. Let $\bM$ be an $F$-rational Levi subgroup of an $F$-rational parabolic subgroup  of $\bG$. Then $p$ satisfies the same assumptions with respect to $\bM$, i.e. $p$ is good for $M$ and does not divide the order of the fundamental groups of $\bM_\der$.

Let $(\bS,\theta)$ be a pair consisting of a tamely ramified torus $\bS$ in $\bM$, and a character $\theta\colon S\to \C^\times$. For any positive real number $r$, consider
\begin{equation} \label{eqn:rootsG0}
\Sigma^{\bS,\theta}_r:=\left\{\gamma\in \Sigma(\bM,\bS)\,:\, (\theta\circ \Nor_{E/F})(\gamma^\vee(E_{r}^\times))=1\right\}.
\end{equation}
We have $\Sigma^{\bS,\theta}_s\subset \Sigma^{\bS,\theta}_r$ for $s<r$. Set $\Sigma^{\bS,\theta}_{r+}:=\bigcap_{s>r}\Sigma^{\bS,\theta}_s$. Then $r\mapsto \Sigma^{\bS,\theta}_r$ defines a $\Gal(F_{\sep}/F)$-invariant filtration. Let $r_{d-1}>r_{d-2}>\cdots>r_0>0$ denote the breaks of the filtration, i.e.~the $r$'s such that $\Sigma^{\bS,\theta}_{r}\ne\Sigma^{\bS,\theta}_{r+}$. We set $r_{-1}:=0$, and let $r_d$ be the depth of $\theta$. We have $r_d\ge r_{d-1}$. For each $i$ such that $0\le i \le d$, we denote by $\bM^i$ the connected reductive subgroup of $\bM$ with maximal torus $\bS$ and root system $\Sigma^{\bS,\theta}_{r_{i-1}+}$. By definition, the root system of $\bM^d$ is $\Sigma(\bM,\bS)$, and thus $\bM^d=\bM$.
The $\bM^i$'s are tame twisted Levi subgroups of $\bM$ by \cite[Lemma~3.6.1]{Kal-reg}. Moreover, the root system of $\bM^0$ is $\Sigma^{\bS,\theta}_{0+}$; if the latter is empty, we have $\bM^0=\bS$. 

Denote $M^i:=\bM^i(F)$. By \cite[Proposition~3.6.7]{Kal-reg}, the pair $(\bS,\theta)$ 
has a Howe factorization with respect to a sequence $(\phi_{-1},\phi_0,\ldots,\phi_d)$ of characters, where $\phi_{-1}\colon S\to\C^\times$ and $\phi_i\colon M^i\to\C^\times$ for $0\le i\le d$. More precisely, we have
\begin{equation}
    \theta=\prod_{i=-1}^d\phi_i,
\end{equation}
\begin{itemize}
    \item for any $i\in\{0,\ldots,d\}$, the character $\phi_i$ is trivial on $(\bM^i)_\scn$, has depth $r_i$ and is $\bM^{i+1}$-generic for any $i\ne d$;
    \item $\phi_d$ is trivial if $r_d=r_{d-1}$, and has depth $r_d$ otherwise.
\end{itemize}

\begin{numberedparagraph}
From now on, we make the following assumptions: $\bS$ is an elliptic maximal torus of $\bM$; the splitting extension of $\bS$ is tamely ramified; $\bS$ is maximally unramified inside $\bM^0$, i.e.~$S$ coincides with its maximal unramified subtorus as in \cite[Definition~3.4.2]{Kal-reg}; $\theta$ is \textit{$k_F$-regular} with respect to $M^0$, in the sense of \cite[Definition~3.1.1]{Kal-ns}.

For any point $x$ in the building of $M$, let $[x]$ be the projection of $x$ onto the reduced building $\mcB_{\red}(M)$. Let $M_x$ (resp.~$M_{[x]}$) be the subgroup of $M$ fixing $x$ (resp.~$[x]$). Recall that $M_{[x]}=\Nor_{M}(M_{x,0})$ by \cite[Lemma~3.3]{Yu}. 
As in \cite[Lemma~3.4.3]{Kal-reg}, we can then associate to $S$ a vertex $[y]$ of $\mcB_{\red}(M)$, which is the unique $\Gal(F^\nr/F)$-fixed point in the apartment $\mcA_{\red}(S,F^{\nr})$ of $\mcB_{\red}(M)$. 

Let $S_\bounded$ be the unique maximal bounded  subgroup of $S$ (which is also the unique maximal compact subgroup of $S$). Denoting by $\fS^\circ$ the connected N\'eron model of $\bS$,  we write  $S_0:=\fS^\circ(\fo_f)\subset S_\bounded$ (see \cite[\S3.1]{Kal-reg} for more details). 
Let
$\bbbM_{y,0}^0$ be the connected reductive $k_F$-group such that 
\begin{equation}
\bbM^0_{y,0}:=\bbbM_{y,0}^0(k_F)=M^0_{y,0}/M^0_{y,0+}.
\end{equation} 
There exists an elliptic maximal $k_F$-torus $\bbbS$ of $\bbbM_{y,0}^0$ such that for every unramified extension $F'$ of $F$, the image of $\bS(F')_0$ in $\bM(F')_{y,0}/\bM(F')_{y,0+}$ is equal to $\bbbS(k_{F'})$ (see \cite[Lemma~3.4.4]{Kal-reg}). 
By \cite[Lemma~3.4.14]{Kal-reg}, the character $\phi_{-1}|_{S_0}$ factors through a \textit{regular}  character $\overline{\phi_{-1}}$ of $\bbS:=\bbbS(k_F)$ as defined in \cite[Definition~3.4.16]{Kal-reg}. In particular, $\overline{\phi_{-1}}$ is in general position in the sense of \cite[Definition~5.15]{DL}. Then it follows from \cite[Proposition~7.4, Theorem~8.3]{DL} that the Deligne-Lusztig character 
\begin{equation} \label{eqn:DL}
    (-1)^{\rankkF(\bbbM_{y,0}^0)-\rankkF(\bbbS)}R_{\bbbS}^{\bbbM_{y,0}^0}(\overline{\phi_{-1}})
\end{equation}
    can be represented by a \textit{cuspidal}, i.e. which cannot be obtained from a
proper parabolic induction, 
$\bbM_{y,0}^0$-module $\bkappa_{S,\phi_{-1}}$, where $\rankkF(?)$ denotes the $k_F$-rank of $?$. Then $\bkappa_{S,\phi_{-1}}$ is  irreducible (see \cite[Definition~5.15]{DL}), and its pull-back to $M^0_{y,0}$ extends uniquely to a representation $\kappa_{S,\phi_{-1}}$ of $S M^0_{y,0}$. We define
\begin{equation} \label{eqn:rho}
\rho_{\bS,\theta}:=\Ind_{SM^0_{y,0}}^{M^0_{[y]}}\kappa_{S,\phi_{-1}}\quad\text{and}\quad
\sigma^0:=\cInd_{M^0_{[y]}}^{M^0}\rho_{\bS,\theta}.
\end{equation}
Then $\sigma^0$ is a depth-zero irreducible regular supercuspidal representation of $M^0$ (see \cite[Definition~3.4.19 \& Proposition~3.4.20]{Kal-reg}). 
Set $\fs_{M^0}:=[M^0,\sigma^0]_{M^0}$. 

More generally,  we define an irreducible supercuspidal representation $\sigma$
of $M$ by using the twisted Yu construction of \cite{Fintzen-Kaletha-Spice}. As observed in \cite[\S3.4]{Kal-ns}, it has the same effect as using the original Yu construction from \cite{Yu} applied to the character  $\theta\cdot\epsilon$, where $\epsilon\colon S\to\{\pm 1\}$ is the product of the characters $\epsilon_y^{M^i/M^{i-1}}$ of \cite[Theorem~3.4]{Fintzen-Kaletha-Spice}. The representation $\sigma$ is regular, i.e.~satisfies \cite[Definition~3.7.13]{Kal-reg}.  
Then $\chi\otimes\sigma$ is regular for any $\chi\in\fX_\nr(M)$, and we say that the inertial class $\fs=[M,\sigma]_G$ is regular.

\end{numberedparagraph}

\begin{numberedparagraph}

For $\fs_M=[M,\sigma]_M$ and $\fs_{M^0}=[M^0,\sigma^0]_{M^0}$, the map
\begin{equation} \label{eqn:map_f}
\begin{matrix}\fff\colon
&\Irr^{\fs_M}(M)&\longrightarrow &\Irr^{{\fs_{M^0}}}(M^0)\cr
&\sigma\otimes\chi&\mapsto&\sigma^0\otimes\chi|_{M^0}
\end{matrix},\quad\chi\in\fX_\nr(M),
\end{equation}
is an isomorphism of varieties by \cite[Theorem~6.1]{Mishra}. Let $\Oo^0$ be the orbit of $\sigma^0$ under the action of $\fX_\nr(M^0)$. Let $W_{G^0}^{\fs^0}=W_{\Oo^0} \rtimes R(\Oo^0)$ be the decomposition analogous to (\ref{eqn:WMO}). Then (\ref{eqn:map_f}) and (\ref{eqn:IrrsM}), applied to both $\fs_M$ and $\fs_{M_0}$, show that the orbits
$\Oo$ and $\Oo^0$ are isomorphic. 
The following is a consequence of \cite[7.3, 9.3]{Ad-Mi}. 
\begin{lem}[Adler-Mishra]\label{lem:GG0}
Suppose that $p$ is good for $G$ and does not divide the order of the fundamental group of $\bG_\der$. Let $\fs=[M,\sigma]_G\in\fB(G)$ be a regular inertial class. Then (1) there is a group isomorphism
\begin{equation}
\fw_\sigma\colon W_G^\fs\longrightarrow W_{G^0}^{\fs^0},
\end{equation}
where $\fs^0=[M^0,\sigma^0]_{G^0}$, 
and $\fff$ is equivariant with respect to $\fw_\sigma$.\\
(2) there is an isomorphism 
\begin{equation}
\fl_\sigma\colon(\Irr^{\fs_M}(M)/\!/W_G^\fs)_\natural\longrightarrow (\Irr^{{\fs_{M^0}}}(M^0)/\!/W_{G^0}^{\fs^0})_{\natural^0}.
\end{equation}
\end{lem} 

The collection $\natural^0$ of $2$-cocycles is defined as follows. For $x\in \Irr^{\fs_M}(M)$, let $W_G^{\fs,x}$ denote the stabilizer of $x$ in $W^\fs$. Since $\fff$ is equivariant with respect to $\fw_\sigma$, the latter restricts to an isomorphism
\begin{equation}
\fw_\sigma|_{W_G^{\fs,x}}\colon W_G^{\fs,x}\longrightarrow W_{G^0}^{\fs^0,\fff(x)},
\end{equation}
and every $2$-cocycle 
$\natural_x\colon W_G^{\fs,x}\times W_G^{\fs,x}\longrightarrow\C^\times$ defines a $2$-cocycle
\begin{equation} \label{eqn:cocycleG0}
\natural_{\fff(x)}^0\colon W_{G^0}^{\fs^0,\fff(x)}\times W_{G^0}^{\fs^0,\fff(x)}\longrightarrow\C^\times.
\end{equation} 
Consequentially, we obtain in Theorem \ref{thm:xiGG0}(2) new cases of \cite[Conjecture~1.1]{Ad-Mi}. 
\begin{theorem} \label{thm:xiGG0}
Suppose  that $p$ is good for $G$ and does not divide the order of the fundamental group of $\bG_\der$. Let $\fs=[M,\sigma]_G\in\fB(G)$ be a regular inertial class. 
\begin{enumerate}
\item[(1)] Then
\begin{equation}
(\xi_{G^0}^{\fs^0})^{-1}\circ\fl_\sigma\circ\xi_G^\fs
\colon\Irr^\fs(G) \longrightarrow \Irr(G^0)^{\fs^0}
\end{equation}
is a bijection.
\item[(2)]
We have a bijection
\begin{equation}
\Irr(\mcH^\fs(G)) \longrightarrow\Irr(\mcH^{\fs^0}(G^0)).
\end{equation}
\end{enumerate}
\end{theorem} 
\begin{proof}
(1) This follows from the fact that the map $\xi_G$ defined in (\ref{eqn:xiG}) and the analogous map 
\begin{equation}
\xi_{G^0}^{\fs^0}\colon   \Irr(G)^{\fs^0} \longrightarrow (\Irr^{\fs_{M^0}}(M^0)/\!/W_{G^0}^{\fs^0})_\natural.
\end{equation}
are isomorphisms.

(2) By \cite[Theorem~9.7]{Solleveld-endomorphism-algebra} applied to both $G$ and $G^0$, we have 
\[\Irr^{\mathfrak{s}}(G)\cong \Irr(\End(i_P^G(V_B))\quad\text{and}\quad 
\Irr^{\mathfrak{s}^0}(G^0)\cong \Irr(\End(i_{P_0}^{G^0}(V_{B^0})).\]
Thus by (1), we have 
$\Irr(\End(i_P^G(V_B))\cong \Irr^{\mathfrak{s}}(G)\cong \Irr^{\mathfrak{s}^0}(G^0)\cong \Irr(\End(i_{P^0}^{G^0}(V_{B^0}))$. 
Then the result follows by applying Corollary~\ref{cor:two_algebras} to both $\fs$ and $\fs^0$.
\end{proof}
\begin{remark} 
The algebras $\mcH^\fs(G)$ and $\mcH^{\fs^0}(G^0)$ are not always isomorphic, as shown in \cite[Example~11.8]{Goldberg-Roche-Hecke-algebras} for $G=\SL_n(F)$. 
However, we show in Theorem~\ref{Weyl_iso} that they are isomorphic when $G=\rG_2$ and $M$ is a maximal Levi subgroup.
\end{remark}
\end{numberedparagraph}

\begin{numberedparagraph}\label{nonsing-gp-side-notations}
We end this section with brief recollections on \textit{non-singular} supercuspidals in the sense of \cite{Kal-ns}, as we will consider \textit{non-singular Bernstein blocks} in \S\ref{non-sing-Bernstein-blocks-property1}. 

Let $\theta\colon \bS(k_F)\to\overline{\Q}_{\ell}$ be a non-singular character. Let $\Nor_\bG(\bS)(k_F)_{\theta}$ denote the stabilizer of the pair $(\bS,\theta)$. By \cite[Proposition~2.3.3]{Kal-ns}, the character $\theta$ extends to the group $\Nor_\bG(\bS)(k_F)_{\theta}$. Let $\bbbU\subset\bG$ be the unipotent radical of a $k_F$-rational Borel subgroup of $\bG$,  containing $\bS$, and let $Y_\bbbU$ be the corresponding Deligne-Lusztig variety. 
 Let $\kappa_{(\bS,\theta)}$ be the isomorphism class of the representation $H_c^{d_\bbbU}(Y_\bbbU,\overline{\Q}_{\ell})_{\theta}$. For simplicity of expositions, we only describe the depth-zero situation. The representation $\pi_{(\bS,\theta)}:=\cInd_{G_x}^{G}\inf^{G_x}_{\bbG_x}
 \kappa_{(S,\theta)}$ is supercuspidal but not necessarily irreducible. When it is indeed irreducible, we return to the case where $\theta$ is regular as in \cite{Kal-reg}. When $\pi_{(\bS,\theta)}$ is reducible (e.g.~when $\theta$ is non-singular non-regular), it decomposes as \cite[3.3.3]{Kal-ns}
\begin{equation}
    \pi_{(\bS,\theta)}=\sum\limits_{\varrho\in \mathrm{Irr}(\Nor_\bG(\bS)(k_F)_{\theta},\theta)}\dim(\varrho)\pi_{(\bS,\theta,\varrho)}^{\epsilon},
\end{equation}
where the constituents $\pi_{(\bS,\theta,\rho)}^{\epsilon}:=\cInd_{G_x}^{G}\inf^{G_x}_{\bbG_x}\kappa_{(\bS,\theta,\varrho)}^{\epsilon}$, constructed from $\kappa_{(\bS,\theta,\varrho)}^{\epsilon}$ as in \cite[Definition 2.7.6]{Kal-ns}, are irreducible \textit{non-singular} supercuspidal representations. Here $\Irr(\Nor_\bG(\bS)(k_F)_{\theta},\theta)$ denotes the set of irreducible representations of $\Nor_\bG(\bS)(k_F)_{\theta}$ whose restriction to $\bS(k_F)$ is $\theta$-isotypic $\epsilon$ is a fixed coherent splitting of the family of $2$-cocycles $\{\eta_{\Psi,\bbbU}\}$ as in \cite[\S 2.4]{Kal-ns}. 
The positive-depth supercuspidals can be described similarly, by applying Yu's construction \cite{Yu} to the representation $\pi_{(G^0,\bS,\phi_{-1})}$ of $G^0$ associated to the pair $(\bS,\phi_{-1})$ \cite[(3.2)]{Kal-ns}.
\end{numberedparagraph}

\section{Local Langlands correspondence for Bernstein blocks} \label{sec:LLC}

\subsection{Axiomatic construction of the correspondence} \label{subsec:Galois} Let $G^\vee$ denote the Langlands dual group of $G$, i.e.~the complex Lie group  with root datum dual to that of $G$.
Let $\rZ_{G^\vee}$ be the center of $G^\vee$ and $G^\vee_{\ad}$ the quotient $G^\vee/\rZ_{G^\vee}$. The $L$-group of $G$ is defined to be ${}^LG:=G^\vee\rtimes W_F$. Similarly, $M^\vee$ denotes the Langlands dual group of $M$. Let $\rZ_{M^\vee\rtimes I_F}$ be the center of $M^\vee\rtimes I_F$, and define
\begin{equation} \label{eqn:Xnrdual}
\fX_\nr({}^LM):=(\rZ_{M^\vee\rtimes I_F})_{W_F}^\circ.
\end{equation}
The group $\fX_\nr({}^LM)$ is naturally isomorphic to the group $\fX_\nr(M)$. We denote the isomorphism $\fX_\nr(M)\xrightarrow{\sim} \fX_\nr({}^LM)$ by $\chi\mapsto\chi^\vee$.

\begin{numberedparagraph}
An $L$-parameter is a continuous morphism
$\varphi\colon W_F'\to {}^LG$ such that 
\begin{itemize}
    \item $\varphi(w)$ is semisimple for each $w\in W_F$;
    \item the restriction $\varphi|_{\SL_2(\C)}$ is a morphism of complex algebraic groups.
\end{itemize}
An $L$-parameter $\varphi$ is said to be \textit{discrete} if $\varphi(W_F')$ is not contained in any proper Levi subgroup of $G^\vee$.
The group $G^\vee$ acts on the set of $L$-parameters. We denote by $\Phi(G)$ the set of $G^\vee$-classes of $G$-relevant $L$-parameters. Attached to each $L$-parameter $\varphi$ for $G$, we define several (possibly disconnected) complex reductive groups as follows. 

Set $\rZ_{G^\vee}(\varphi):=\rZ_{G^\vee}(\varphi(W'_F))$. Let   $\rZ_{G^\vee_\sconn}^1(\phi)$ be the inverse image of $\rZ_{G^\vee}(\phi)/\rZ_{G^\vee}(\phi)\cap \rZ_{G^\vee}$ (viewed as a subgroup of $G^\vee_\ad$) under the quotient map $G^\vee_\sconn\to G^\vee_{\ad}$. Then we set
\begin{equation} \label{eqn:Gphi}
 \mcG_\varphi:=\rZ_{G^\vee}^1(\varphi|_{W_F}).
\end{equation} 
We also define the following component group
\begin{equation} \label{eqn:A-phi}
\mcS_\varphi:=\rZ_{G^\vee}^1(\varphi)/\rZ_{G^\vee}^1(\varphi)^\circ.
\end{equation}
An \textit{enhancement} of $\varphi$ is an irreducible representation $\varrho$ of $S_\varphi$. Pairs $(\varphi,\varrho)$ are called \textit{enhanced $L$-parameters} (for $G$ and its inner forms). 
\end{numberedparagraph}
\begin{numberedparagraph}
Let $G^\vee$ act on the set of enhanced $L$-parameters  via
\begin{equation} \label{eqn:Gvee-action}
g\cdot (\varphi,\varrho) = (g \varphi g^{-1},g \cdot \varrho).
\end{equation}
We denote by $\Phi_\enh(G)$ the set $G^\vee$-conjugacy classes of enhanced $L$-parameters.
We define an action of $\fX_\nr({}^LM)$ on $\Phi_\enh(M)$ as follows. 
Given $(\varphi,\varrho)\in\Phi_\enh(M)$ and $\xi\in \fX_\nr({}^LM)$, we define $(\xi\varphi,\varrho)\in\Phi_\enh(M)$ by 
$\xi\varphi:=\varphi$ on $I_F\times\SL_2(\C)$ and  $(\xi\varphi)(\Frob_F):=\tilde \xi \varphi(\Frob_F)$. Here $\tilde \xi\in \rZ_{M^\vee\rtimes I_F}^\circ$ represents $z$.

For an $L$-parameter $\varphi$ of $G$, we denote by $u_\varphi$ the image of 
$\left(1,\left(\begin{smallmatrix} 1&1\cr 0&1\end{smallmatrix}\right)\right)$ under $\varphi$. By \cite[(92)]{AMS1}, we have $u_\varphi\in \mcG_\varphi^\circ$ and
\begin{equation}\mcS_\varphi\simeq\rZ_{\mcG_\varphi}(u_\varphi)/\rZ_{\mcG_\varphi}(u_\varphi)^\circ:=A_{\mcG^\circ_\varphi}(u_\varphi).
\end{equation} 
Let $\varrho$ be an irreducible representation of $A_{\mcG^\circ_\varphi}(u_\varphi)$. The pair $(u_\varphi,\varrho)$ is said to be   \textit{cuspidal} if there is an $\mcG^\circ$-equivariant \textit{cuspidal} (in the sense of Lusztig \cite{Lusztig-IC}) local system on the $\mcG^\circ_\varphi$-conjugacy class of $u_\varphi$. 
\begin{Defn} \label{def:cuspidal}
An enhanced $L$-parameter $(\varphi,\varrho)$ is said to be \textit{cuspidal} if $\varphi$ is discrete and $(u_\varphi,\varrho)$ is a cuspidal pair in $\mcG_\phi$.
\end{Defn}
\end{numberedparagraph}

\begin{numberedparagraph}
From now on, we use the subscript $\cus$ to denote ``cuspidal''. Let $(\varphi_\cus,\varrho_\cus)$ be a cuspidal enhanced $L$-parameter for $M$, we denote by
\begin{equation} \label{eqn:svee}
\fs^\vee:=[{}^LM,\varphi_\cus,\varrho_\cus]_{G^\vee}   
\end{equation}
the $G^\vee$-conjugacy class of $({}^LM,\Oo^\vee)$, where $\Oo^\vee$ is  the orbit of $(\varphi_\cus,\varrho_\cus)$ under the action of $\fX_\nr({}^LM)$. Let $\fB^\vee(G)$ be the set of such $\fs^\vee$. Set  $\fs^\vee_{M}:=[{}^LM,\varphi_\cus,\varrho_\cus]_{M^\vee}$. Let 
\begin{equation}
\Nor^{\fs^\vee}_{G^\vee}:=\left\{n\in \Nor_{G^\vee}(M^\vee)\,:\, {}^n(\varphi_\cus,\varrho_\cus)\simeq (\varphi_\cus,\varrho_\cus)\otimes\chi^\vee\;\text{ for some $\chi^\vee\in\fX_\nr({}^LM)$}\right\}.
\end{equation}
Denote 
\begin{equation} \label{eqn:Wsdual}
    W^{\fs^\vee}_{G^\vee}:=W(M^\vee,\Oo^\vee):=\Nor^{\fs^\vee}_{G^\vee}/M^\vee.
\end{equation}
The group $W_{G^\vee}^{\fs^\vee}$ is a finite extended Weyl group, i.e.~it decomposes as
\begin{equation} \label{eqn:decWdual}
W_{G^\vee}^{\fs^\vee}= W_{\Oo^\vee} \rtimes R(\Oo^\vee),
\end{equation}
where $W_{\Oo^\vee}$ is a finite Weyl group, and $R(\Oo^\vee)$ is a finite abelian group (see \eqref{defining-ROcheck}).
\end{numberedparagraph}
\begin{numberedparagraph}
Let $\Phi_\enh^\cusp(M)$ denote the set of $M^\vee$-conjugacy classes of cuspidal enhanced $L$-parameters for $M$.
By \cite[(115)]{AMS1}, there is a decomposition of $\Phi_\enh(G)$ into series of enhanced $L$-parameters indexed by the set $\fB^\vee(G)$: 
\begin{equation} \label{eqn:partition}
\Phi_\enh(G) =\bigsqcup_{\fs^\vee\in \fB^\vee(G)}\Phi_\enh^{\fs^\vee}(G), 
\end{equation}
where $\Phi_\enh^{\fs^\vee}(G)$ consists of enhanced $L$-parameters whose \textit{cuspidal support} lies in $\fs^{\vee}$. Moreover, for any $(\varphi_\cus,\varrho_\cus)\in\Phi_\enh^\cusp(M)$, we have
\begin{equation} \label{eqn:cusp}
\Phi_\enh^{\fs^\vee_M}(M)=\fX_\nr({}^LM)\cdot(\varphi_\cus,\varrho_\cus).
\end{equation}
For $\fs^\vee=[{}^LM,\varphi_\cus,\varrho_\cus]_{G^\vee}\in \fB^\vee(G)$, there exists a bijection \cite[Theorem~9.3]{AMS1}
\begin{equation} \label{eqn:xiGv}
  \xi_{G^\vee}^{\fs^\vee}\colon\Phi_\enh^{\fs^\vee}(G)\longrightarrow (\Phi_\enh^{\fs_M^\vee}(M)/\!/W^{\fs^\vee}_{G^\vee})_{{}^L\natural},
\end{equation}
where ${}^L\natural=({}^L\natural_z)_{z\in \fX_\nr({}^LM)/\fX_\nr({}^LM,(\varphi_\cusp,\rho_\cus))}$ is a collection of $2$-cocycles, and 
\begin{equation} \label{eqn:dual stab}
\fX_\nr({}^LM,(\varphi_\sigma,\varrho_\sigma)):=\left\{z\in \fX_\nr({}^LM)\,:\,z\cdot(\varphi_\sigma,\varrho_\sigma))=(\varphi_\sigma,\varrho_\sigma)\right\}.
\end{equation}
\end{numberedparagraph}
\begin{numberedparagraph}
Let $\eta\colon M\to \widetilde M$ be an $F$-morphism of connected reductive $F$-groups with abelian kernel and cokernel. Then by \cite[\S1.4]{Bor}, $\eta$ induces a map from the root datum of $M^\vee$ to that of $M^\vee$. We denote by $\eta^\vee\colon \widetilde M^\vee\to M^\vee$ the associated morphism of algebraic groups as in \cite[\S2.1]{Bor}. 
\begin{align} \label{eqn:Leta}
{}^L\eta\colon& {}^L\widetilde M\to{}^LM\cr
&(\widetilde m,w)\mapsto (\eta^\vee(\widetilde m),w)\;\;\text{for $\widetilde m\in M^\vee$.}
\end{align}
We recall \cite[Desideratum 10.3(5)]{Bor}: 
Let $\widetilde\varphi\colon W_F\times W_F\to \widetilde M^\vee$ be an $L$-parameter  for $\widetilde M$ and set $\varphi:={}^L\eta\circ \widetilde \varphi$. Then for any $\widetilde\pi\in\Pi_{\widetilde\varphi}$, the representation $\widetilde\pi \circ \eta$ is the direct sum of finitely many irreducible representations belonging to $\Pi_\varphi$.
\end{numberedparagraph}

\begin{numberedparagraph}
Let $\Irr^\scusp(M)\subset \Irr(M)$ denote the set of equivalence classes of irreducible supercuspidal representations of $M$. Let $g\in\Nor_\bG(\bM)(F)$. For any $\sigma\in\Irr^\scusp(M)$, we have ${}^g\sigma\in\Irr^\scusp(M)$. We denote by $c_g$ the isomorphism 
\begin{equation} \label{eqn:c_g}
c_g\colon (M,\sigma)\overset{\sim}{\to} (M,{}^g\sigma).
\end{equation}
Let ${}^Lc_g\colon {}^LM\to {}^LM$ be the morphism defined by \eqref{eqn:Leta}.
Let $w\mapsto w^\vee$ be the canonical isomorphism from $W(M):=\Nor_\bG(\bM)(F)/M$ to $W(M^\vee):=\Nor_{G^\vee}(M^\vee)/M^\vee$ defined in \cite[Proposition~3.1]{ABPS-CM}. Let $n_w$ (resp.~$n_{w^\vee}$) be a representative of $w$ (resp.~$w^\vee$) in $\Nor_\bG(\bM)(F)$ (resp.~$\Nor_{G^\vee}(M^\vee)$). 
With these notations, we have  $c_{n_w}^\vee(m^\vee)={}^{n_{w^\vee}}m^\vee$.  

For $\sigma\in \Irr(M)$, since the equivalence class of ${}^{n_w}\sigma$ does not depend on the choice of the representative $n_w$, we will simply denote it by ${}^w \sigma$. Similarly, we use the notation ${}{w^\vee}$, instead of ${}{n_{w^\vee}}$, to denote the action of $\Nor_{G^\vee}(M^\vee)$ on $\Phi_\enh(M)$.
\begin{property} \label{property}
Let $M$ be a Levi subgroup of $G$. Let  $\fs_M:=[M,\sigma]_M\in\fB(M)$. There exists a map
\begin{equation}
    \begin{matrix}\fL^{\fs_M}\colon&\Irr^{\fs_M}(M)&\xlongrightarrow&\Phi_\enh^\cusp(M)\cr
    &\sigma&\mapsto &(\varphi_\sigma,\varrho_\sigma)
    \end{matrix}
\end{equation}
such that the following properties are satisfied for any $\sigma\in\Irr^{\fs_M}(M)$:
\begin{enumerate}
    \item[(1)]\label{property-1} For any $\chi\in\fX_\nr(M)$, we have
    \begin{equation}    (\varphi_{\chi\otimes\sigma},\varrho_{\chi\otimes\sigma})=\chi^\vee\cdot(\varphi_\sigma,\varrho_\sigma),
    \end{equation}
    where $\chi\mapsto\chi^\vee$ is the canonical isomorphism $\fX_\nr(M)\overset{\sim}{\to}\fX_\nr({}^LM)$.
    \item[(2)]\label{property-2}
    For any $w\in W(M)$, we have
 \begin{equation} {}^{w^\vee}(\varphi_\sigma,\varrho_\sigma)\simeq (\varphi_{{}^w\sigma},\varrho_{{}^w\sigma}),
 \end{equation}
 where $w\mapsto w^\vee$ is the canonical isomorphism  $W(M)\overset{\sim}{\to}W(M^\vee)$.
    \end{enumerate}
\end{property}
\end{numberedparagraph}

\begin{remark} \label{rem:Desiderata}
\begin{enumerate}
\item 
Property~\ref{property}(1) is closely related to  \cite[Desideratum~10.3.(2)]{Bor}. 
\item
Property~\ref{property}(2) is a stronger version of \cite[Desideratum~10.3.(5)]{Bor} for $\eta=c_g$, and can be viewed as an analogue of \cite[Conjecture~5.2.4]{Haines} for enhanced $L$-parameters for supercuspidal representations. In the special case where the $L$-packet of $\sigma$ is a singleton, Property~\ref{property}(2) is in fact equivalent to \cite[Desideratum~10.3.(5)]{Bor} for $\eta=c_g$.
\end{enumerate}
\end{remark}

\begin{lem} \label{lem:duals}  Let $\fs=[M,\sigma]_M\in \fB(G)$ satisfy {\rm Property}~\ref{property}. 
Then there is a group isomorphism
\begin{equation}\label{defining-frak-r}
\fr\colon W_G^{\fs}\overset{\sim}{\rightarrow}  W_{G^{\vee}}^{\fs^{\vee}},\quad\text{where $\fs^\vee:=[{}^LM,\fL^{\fs_M}(\sigma)]_{G^\vee}$}.
\end{equation}
Moreover, $\fL^{\fs_M}$ is equivariant with respect to $\fr$.
\end{lem}
\begin{proof}
Let $w\in W_G^{\mathfrak{s}}\subset W(M)$. By the definition of $W^\fs_G$, we have ${}^{w}\sigma\simeq \chi\otimes\sigma$ for some $\chi\in\fX_\nr(M)$.  By 
Property~\ref{property}, we have
\begin{equation}
{}^{w^\vee}(\varphi_\sigma,\varrho_\sigma)\simeq (\varphi_{{}^w\sigma},\varrho_{{}^w\sigma})\simeq
(\varphi_{\chi\otimes\sigma},\varrho_{\chi\otimes\sigma})\simeq \chi^\vee\cdot(\varphi_\sigma,\varrho_\sigma).
\end{equation}
Thus $w^\vee\in W_{G^{\vee}}^{\fs^{\vee}}$, and the map $w\mapsto w^\vee$ gives a group morphism from $W_G^{\fs}$ to $W_{G^{\vee}}^{\fs^{\vee}}$. Reversing the argument, we see that it is an isomorphism.
\end{proof}

We suppose that Property~\ref{property} holds, and that the collections of $2$-cocycles $\natural$ and ${}^L\natural$ satisfy the following
\begin{equation} \label{eqn:matching cocycles}
    {}^L\natural_{\chi^\vee}=\natural_{\chi}\quad\text{for any $\sigma\in\fs$ and any $\chi\in \fX_\nr(M)/\fX_\nr(M,\sigma)$},
\end{equation}

\begin{theorem} \label{thm:matching-extended-qts}
\begin{enumerate}
    \item We have a canonical isomorphism
\begin{equation}
\fe\colon(\Irr^{\fs_M}(M)/\!/W_G^\fs)_{\natural}\overset{\sim}{\longrightarrow}(\Phi_\enh^{\fs_M^\vee}(M)/\!/W^{\fs^\vee}_{G^\vee})_{{}^L\natural}.
\end{equation}
\item 
There is a bijection 
\begin{equation}
\fL:=(\xi_{G^\vee}^{\fs^\vee})^{-1}\circ\fe\circ\xi_G^\fs\colon
\Irr^\fs(G)\longrightarrow\Phi_\enh^{\fs^\vee}(G).
\end{equation}
\end{enumerate}
\end{theorem}
\begin{proof}
Let $\fs=[M,\sigma]_G$. Consider $(\varphi_\sigma,\varrho_\sigma):=\fL^{\fs_M}(\sigma)$. Then the isomorphism $\fX_\nr(M)\simeq\fX_\nr({}^LM)$ 
combined with (\ref{eqn:cusp}) shows that
\begin{equation}
\Phi_\enh^{\fs^\vee_M}(M)=\left\{\chi^\vee\otimes (\varphi_\sigma,\varrho_\sigma)\,:\,\chi \in \fX_\nr(M)\right\}.
\end{equation}
By Property~\ref{property}(1) and \eqref{eqn:IrrsM}, we have 
\begin{equation}
\Phi_\enh^{\fs^\vee_M}(M)=\left\{(\varphi_{\chi\otimes\sigma},\varrho_{\chi\otimes\sigma})\,:\,\chi \in \fX_\nr(M)\right\}\simeq
\Irr^{\fs_M}(M).
\end{equation}
Recall that $W_G^{\fs,\chi\otimes\sigma}$ denotes the stabilizer of $\chi\otimes\sigma$ in $W_G^\fs$. By construction of the extended quotients in \eqref{twisted-extended-quotient}, we have 
\begin{equation}\label{extended-quotient-group-side}
(\Irr^{\fs_M}(M)/\!/W_G^\fs)_\natural=\bigsqcup_{\chi\in\fX_\nr(M)/\fX_\nr(M,\sigma)}(\Irr(\C[W_G^{\fs,\chi\otimes\sigma},\natural_{\chi}])/W_G^\fs.
\end{equation}
For $\chi\in\fX_\nr(M)$, let $W_{G^\vee}^{\fs^\vee,\chi^\vee\cdot(\varphi_\sigma,\varrho_\sigma)}$ denote the stabilizer of $\chi^\vee\cdot(\varphi_\sigma,\varrho_\sigma)$ in $W_{G^\vee}^{\fs^\vee}$. 
By Property~\ref{property}(1), we have
\begin{equation}
\fX_\nr(M^\vee,(\varphi_\sigma,\varrho_\sigma))\simeq \fX_\nr(M,\sigma).
\end{equation}
Again by \eqref{twisted-extended-quotient}, we have
\begin{equation}\label{extended-quotient-Galois-side}
(\Phi_\enh^{\fs_M^\vee}(M)/\!/W^{\fs^\vee}_{G^\vee})_{{}^L\natural}=\bigsqcup_{\chi^\vee\in\fX_\nr(M^\vee)/\fX_\nr(M^\vee,(\varphi_\sigma,\varrho_\sigma))}\Irr(\C[W_{G^\vee}^{\fs^\vee,\chi^\vee\cdot(\varphi_\sigma,\varrho_\sigma)},{}^L\natural_{\chi^\vee}])/W^{\fs^\vee}_{G^\vee}.
\end{equation}
For any $w\in W_G^{\fs,\chi\otimes\sigma}$,  by Property~\ref{property}, we have
\begin{equation}
{}^{w^\vee}(\chi^\vee\cdot(\varphi_\sigma,\varrho_\sigma))=
{}^{w^\vee}(\varphi_{\chi\otimes\sigma},\varrho_{\chi\otimes\sigma})=(\varphi_{{}^w(\chi\otimes\sigma)},\varrho_{{}^w(\chi\otimes\sigma)}=(\varphi_{\chi\otimes\sigma},\varrho_{\chi\otimes\sigma}).
\end{equation}
Thus the morphism $\fr$ from \eqref{defining-frak-r} restricts to an isomorphism 
\begin{equation}
    W_G^{\fs,\chi\otimes\sigma}\xrightarrow{\sim}W_{G^\vee}^{\fs^\vee,\chi^\vee\cdot(\varphi_\sigma,\varrho_\sigma)}.
\end{equation}
Combined with \eqref{extended-quotient-group-side} and \eqref{extended-quotient-Galois-side}, we obtain an isomorphism
\begin{equation}
\fe\colon(\Irr^{\fs_M}(M)/\!/W_G^\fs)_\natural\longrightarrow(\Phi_\enh^{\fs_M^\vee}(M)/\!/W^{\fs^\vee}_{G^\vee})_{{}^L\natural}.
\end{equation} 
Then (2) follows from the combination of (1) with  Proposition~\ref{prop:bijectionsS} and \eqref{eqn-xiGv}.
\end{proof}
\begin{remark} \label{rem:Mou}
When $G$ is a split classical group, $F$ has characteristic zero and $\fL^{\fs_M}$ is the LLC defined by Arthur in \cite{Arthurbook}, then  Lemma~\ref{lem:duals} was proved in \cite[Theorem~4.1]{Mou}, and Theorem~\ref{thm:matching-extended-qts} follows from \cite[\S3.2 \& 3.3]{Moussaoui-CM}.
\end{remark}
\begin{remark} \label{rem:triviality cocycles}
The $2$-cocycles in $\natural$ and ${}^L\natural$ are expected to be often trivial:
(1) They are trivial if the groups $R(\Oo)$ from \eqref{eqn:WMO} and $R(\Oo^\vee)$ from \eqref{eqn:decWdual} have cardinality at most $2$, and hence when $M$ is a Levi subgroup of a maximal parabolic subgroup of $G$. 

(2) The $2$-cocycles are also trivial when $G$ is a symplectic group or a split special orthogonal group  by \cite[Theorem~7.7]{Heiermann-intertwining-operators-Hecke-algebras} on the group side and \cite[\S4.5]{Moussaoui-CM} on the Galois side. They are trivial for principal series representations of split groups by \cite[Corollary~7.9 and Theorem~8.2]{Roche-Hecke-algebra} on the group side and \cite[Theorem~13.1]{ABPS-disc} on the Galois side. 

(3) However, there exist cases when these $2$-cocycles are not trivial: e.g.~see \cite[Example~5.5]{ABPS-Jussieu} for an example where $\natural$ is non-trivial, and \cite[Example~9.4]{AMS1} for an example where ${}^L\natural$ is non-trivial.
\end{remark}

\begin{numberedparagraph}
Let $\sigma$ be a regular supercuspidal irreducible representation of $M$. Let $\varphi_\sigma\colon W_F\to {}^LM$ be the $L$-parameter of $\sigma$ as constructed in \cite{Kal-ns}. 
\begin{prop} \label{thm:LLC-singleton} 
Let $\fs=[M,\sigma]_G$ be such that the $L$-packet for $\sigma$ is a singleton. 
Let $\fs^\vee=[M^\vee,(\varphi_\sigma,1)]_{G^\vee}$. Assume that the collections of $2$-cocycles $\natural$ and ${}^L\natural$ are both trivial. 
We have the following bijection 
\begin{equation}
\fL\colon \Irr^\fs(G)\xlongrightarrow{}\Phi_\enh^{\fs^\vee}(G).
\end{equation}
\end{prop}
\begin{proof}
Let $\fL^{\fs_M}$ be the map
\begin{equation}
\fL^{\fs_M}\colon\sigma\mapsto (\varphi_\sigma,1).
\end{equation}
By Proposition~\ref{nonsing-property-1}, Property~\ref{property}(1) is satisfied.

The validity of \cite[Desideratum 10.3(5)]{Bor} has been established in \cite[Theorem~1.1]{Bourgeois-Mezo} when the $L$-parameter $\widetilde\varphi$ is supercuspidal, and $G$ is quasi-split. Since the  
$L$-packet for $\sigma$ is a singleton, 
 Property~\ref{property}~(2) holds by Remark~\ref{rem:Desiderata}. 
 The result thus follows from Corollary \ref{thm:matching-extended-qts}.
\end{proof}
\begin{remark}
Proposition \ref{thm:LLC-singleton} is sufficient for the case of $G=\rG_2$ (as also exemplified in \cite{AX-LLC}), since $M$ is either $\GL_2$ or a torus, both only having singleton $L$-packets for their supercuspidals. 
\end{remark}
\end{numberedparagraph}

\begin{numberedparagraph}\label{non-sing-Bernstein-blocks-property1}
Suppose now $\sigma$ is a \textit{non-singular} supercuspidal irreducible representation of $M$. Let $\varphi_{\sigma}\colon W_F\to {}^LM$ be the $L$-parameter of $\sigma$ defined in \cite[\S 4.1]{Kal-ns}, with enhancement $\varrho_{\sigma}$. 

More precisely, consider the $\pi_{(\bS,\theta,\varrho)}^{\epsilon}$ recalled in $\S$\ref{nonsing-gp-side-notations}. We fix a coherent splitting $\epsilon$. We recall the construction of the non-singular $L$-parameter $\varphi_{(\bS,\theta)}$ in \cite[\S 4.1]{Kal-ns} (we drop $\varrho$ from the notation as the $L$-parameter does not depend on $\varrho$), which is given as a composition $W_F\xrightarrow{\varphi_S} {}^LS\xrightarrow{{}^Lj}{}^LG$. Here $\varphi_S\colon W_F\to {}^LS$ is Langlands parameter for the character $\theta$, and ${}^Lj\colon{}^LS\to {}^LG$ is a certain $L$-embedding arising from the $\chi$-data as part of some \textit{torally wild $L$-packet datum} $(\bS,\widehat{j},\chi_0,\theta)$ as \textit{loc.cit.}. 
\begin{prop}\label{nonsing-property-1}
The map $\mathfrak{L}^{\fs_M}\colon\sigma\mapsto (\varphi_{\sigma},\varrho_{\sigma})$ satisfies Property \ref{property}~(1).  
\end{prop}
\begin{proof}
By \cite[Proposition 3.4.6]{Kal-ns}, we have $\chi\otimes\pi_{(\bS,\theta,\varrho)}^{\epsilon}=\pi_{(S,\chi\cdot\theta,\chi\otimes\rho)}^{\epsilon}$. On the other hand, by \cite[Proposition 4.5.3]{Kaletha-epipelagic}, we have $\varphi_{(\bS,\chi\cdot\theta)}=(\varphi_{\chi}\cdot\varphi_S)\circ{}^Lj$ (since the $L$-embedding ${}^Lj$ stays the same after twisting by $\chi$), where $\varphi_{\chi}\colon W_F\to \rZ_{\widehat{G}}$ is the corresponding $1$-cocycle as defined \textit{loc.cit.} 
\end{proof}

\begin{remark}\label{nonsing-property-2}
Property \ref{property} (2) of the map $\mathfrak{L}^{\fs_M}\colon\sigma\mapsto (\varphi_{\sigma},\varrho_{\sigma})$ follows from \cite[Conjecture 2.12]{Kaletha-ICM-2022}, which is expected to hold for LLC for non-singular supercuspidal representations. The authors intend to return to this question in future work. 
\end{remark}

\end{numberedparagraph}

\subsection{Matching of simple modules of extended affine Hecke algebras} \label{subsec:EAH}
In this section, we use Corollary~\ref{thm:matching-extended-qts} to obtain a bijection between simple modules of extended affine Hecke algebras for the $p$-adic group and Galois sides, assuming Property~\ref{property}.
Note that this is a completely reasonable assumption, as many groups satisfy this property. Therefore, our main results give a new approach to constructing local Langlands correspondances ``inductively'', building from LLC's on the Levi subgroups'.

\begin{numberedparagraph}
We now recall the construction of a (possibly twisted) extended affine Hecke algebra $\mcH^{\fs^\vee}(G^\vee)$ constructed in \cite{AMS3}. Consider
\begin{equation} \label{eqn:Mphi}
\mcM_{\varphi_{\cus}}:=\Cent_{M^\vee}(\varphi_\cus(W_F)).
\end{equation} 
Let $\mcA_{\varphi_\cus}$ be the identity component of the center of $\mcM_{\varphi_\cus}$. We set
\begin{equation} \label{eqn:Jphi}
\mcJ_\varphi=\mcJ_\varphi^{G^\vee}:=\rZ_{G^\vee}(\varphi(I_F)).
\end{equation}
Let $\Sigma(\mcJ_\varphi^\circ,\mcA_{\varphi_\cus})$ be the set of $\alpha \in X^*(\mcA_{\varphi_\cus}) \setminus \{0\}$ which appear in the adjoint action of $\mcA_{\varphi_\cus}$ on the Lie algebra of $\mcJ_\varphi^\circ$. It is a root system by \cite[Proposition~3.9]{AMS3}. Let $\Sigma(\mcJ_\varphi^\circ,\mcA_{\varphi_\cus})^+$ be the positive root system defined by an $F$-rational Borel subgroup  of $\mcJ_\varphi^\circ$. 
Let $\Delta$ be a basis of the reduced root system $\Sigma(\mcJ_\varphi^\circ,\mcA_{\varphi_\cus})_\red$. 
Let  $a\in\mcA_{{\varphi_\cus}}$ be such that $\alpha(a^{-1})$ is an eigenvalue of $\Ad(\varphi(\Frob))$ for any $\alpha\in\Delta$. We define $\varphi_a\in\Phi(M)$ by
\begin{equation} \label{eqn:varphi_t}
\varphi_a|_{I_F\times\SL_2(\C)}:=\varphi_\cus|_{I_F\times\SL_2(\C)}\quad\text{and}\quad
\varphi_a(\Frob_F):=a\cdot\varphi_\cus(\Frob_F).
\end{equation}
By \cite[Proposition~3.9]{AMS3}, we have $\Sigma(\mcG_{\varphi_a},\mcA_{{\varphi_\cus}})_\red=
\Sigma(\mcJ_\varphi^\circ,\mcA_{\varphi_\cus})_\red$, where  $\mcG_{{\varphi_a}}=\Cent_{G^\vee}(\varphi_a(W_F))$. Consider
\begin{equation} \label{eqn:unrset}
\fX_\nr(M^\vee,\varphi_a):=\left\{x\in \fX_\nr(M^\vee)\,:\,(z\varphi_a)_{M^\vee}=(\varphi_a)_{M^\vee}\right\}.
\end{equation}
Set $T_{\Oo^\vee}:=\fX_\nr(M^\vee)/\fX_\nr(M^\vee,\varphi_a)$. For each $\alpha\in \Sigma(\mcJ_\varphi^\circ,\mcA_{\varphi_\cus})_\red$, let $m_\alpha\in\Z_{>0}$ be the smallest integer such that
\begin{equation}
\ker(m_\alpha\alpha)\supset \left\{a' \in \mcA_{{\varphi_\cus}}\,:\, (a'\varphi_a)_{M^\vee}=(\varphi_a)_{M^\vee}\right\}.
\end{equation}
We set 
\begin{equation} \label{eqn:ROdual}
\Sigma_{\Oo^\vee}:=\left\{m_\alpha\,:\, \alpha \in \Sigma(\mcJ_\varphi^\circ,\mcA_{\varphi_\cus})_\red\right\}\subset X^*(T_{\Oo^\vee}).
\end{equation}
Then $W_{\Oo^\vee}$ from \eqref{eqn:decWdual} is the finite Weyl group of $\Sigma_{\Oo^\vee}$, and
\begin{equation}\label{defining-ROcheck} 
R(\Oo^\vee): = 
\left\{ w \in W(M^{\vee},\Oo^\vee)  \,:\, w \cdot \Sigma(\mcJ_\varphi^\circ,\mcA_{\varphi_\cus})^+ = \Sigma(\mcJ_\varphi^\circ,\mcA_{\varphi_\cus})^+\right\}.
\end{equation}
Let 
\begin{equation}
\lambda^\vee\colon \Sigma_{\Oo^\vee}\to \Z_{\geq 0}\quad\text{and}\quad \lambda^{* \vee} \colon \left\{m_\alpha\alpha\in\Sigma_{\Oo^\vee}:(m_\alpha\alpha)^\vee\in 2X_*(T_{\fs^\vee})\right\}\to \Z_{\geq 0}
\end{equation}
be the two parameter functions defined in the proof of \cite[Lemma~3.12]{AMS3}.
Recall from \cite[(28)]{AMS3} that $\lambda^{* \vee}(\alpha)=\lambda^\vee(\alpha)$ unless $\alpha$ is a short root in a type $B$ root subsystem of $R_{\Oo^\vee}$.

The algebra $\mcH^{\fs^\vee}(G^\vee)$ is defined to be
\begin{equation} \label{eqn:dualHeckealgebra}
\mcH^{\fs^\vee}(G^\vee):=
\mcH_{\mathrm{aff}}(\Oo^\vee,\Sigma_{\Oo^\vee},\lambda^\vee,\lambda^{\vee *},z)\rtimes \C[R(\Oo^\vee),\kappa^\vee],
\end{equation}
where 
$z$ is a positive real number,
$\mcH_{\mathrm{aff}}(G^\vee,\fs^\vee):=\mcH_{\mathrm{aff}} (\Oo^\vee,\Sigma_{\Oo^\vee},\lambda^\vee ,\lambda^{* \vee},z)$ is
the corresponding affine Hecke algebra with affine Weyl group $W_{\Oo^\vee}\ltimes X^*(T_{\Oo^\vee})$, and $\kappa^\vee$ a $2$-cocycle on $R(\Oo^\vee)$.
\end{numberedparagraph}

\begin{theorem}\label{matching-gp-Galois-sides}
We suppose that the collections of $2$-cocycles $\natural$ and ${}^L\natural$ are both trivial.
Let $\fs=[M,\sigma]_G$ and  $\fs^\vee=[{}^LM,\varphi_\sigma,1]_{G^\vee}$. There is a bijection
\begin{equation}
\Irr(\mcH^\fs(G))\,\longleftrightarrow\,\Irr(\mcH^{\fs^\vee}(G^\vee)).
\end{equation}
\end{theorem}
\begin{proof}
By Corollary~\ref{cor:two_algebras}, we have a bijection
\begin{equation}\label{eqn-two-algebras}
\Irr(\mcH^\fs(G))\longrightarrow(\Irr^{\fs_M}(M)/\!/W_G^\fs)_\natural.
\end{equation}
By Corollary~\ref{thm:matching-extended-qts}, we have
\begin{equation}\label{eqn-lem-duals}
\Irr^{\fs_M}(M)/\!/W_G^\fs\simeq \Phi_\enh^{\fs_M^\vee}(M)/\!/W^{\fs^\vee}_{G^\vee}.
\end{equation}
On the other hand, by (\ref{eqn:xiGv}) we have a bijection 
 \begin{equation}\label{eqn-xiGv} 
 \Phi_\enh^{\fs_M^\vee}(M)/\!/W^{\fs^\vee}_{G^\vee}\longrightarrow \Phi_\enh^{\fs^\vee}(G).
 \end{equation} 
Finally, by \cite[Theorem~3.18]{AMS3}, there is a bijection 
\begin{equation}\label{eqn-AMS3-3.18}
\Phi_\enh^{\fs^\vee}(G)\xlongrightarrow{\sim} \Irr(\mcH^{\fs^\vee}(G^\vee)).
\end{equation}
Combining equations \eqref{eqn-two-algebras}, \eqref{eqn-lem-duals}, \eqref{eqn-xiGv} and \eqref{eqn-AMS3-3.18}, we obtain the desired bijection. 
\end{proof}

\begin{Coro} With the same assumption as in Theorem \ref{matching-gp-Galois-sides}. 
\[\Irr(\mcH^{\fs^{\vee}}(G^{\vee})) \longrightarrow\Irr(\mcH^{(\fs^0)^{\vee}}((G^0)^{\vee})).\]
\end{Coro}
\begin{proof}
This follows from combining Theorems \ref{matching-gp-Galois-sides} and \ref{thm:xiGG0}. 
\end{proof}

\section{Applications to \texorpdfstring{$\rG_2$}{\rG_2}} \label{sec:G2}
In this section, we introduce notations and background specifically for the $\rG_2$ case.
\subsubsection{\textbf{General background}}
\begin{numberedparagraph}
Let $\fa_M$ be the real Lie algebra of $A_\bM$, and $\fa_M^*$ its dual. Let $\fa_{M,\C}^*$ be the complexification of $\fa_M^*$. Let $|\cdot|_F$ be the modulus of $F$. Let $H_M: M\to \fa_M$ be such that $q^{-\langle H_M(m),\alpha\rangle}=|\alpha(m)|_F$ for every rational character $\alpha$ of $M$ and every $m\in M$. Note that the kernel of $H_M$ is equal to $M_1$ (recall from $\mathsection$\ref{defining-M1}). 
Consider
\begin{equation} \label{eqn:Msigma}
M_\sigma:=\bigcap_{\chi\in \fX_\nr(M,\sigma)}\ker \chi,
\end{equation}
which has finite index in $M_1$. Recall $\fX_\nr(M,\sigma)$ from \eqref{eqn:XnrMsigma}, and we have  
\begin{equation}
\Irr(M_\sigma/M_1)\simeq \fX_\nr(M)/\fX_\nr(M,\sigma)
\quad\text{and}\quad
\C[M_\sigma/M_1]\simeq \C[\fX_\nr(M)/\fX_\nr(M,\sigma)].
\end{equation}
Set $(M_\sigma/M_1)^\vee:=\Hom_\Z(M_\sigma/M_1,\Z)$. Composing $H_M$ with the $\R$-linear extension of $H_M(M_\sigma/M_1)\to \Z$ gives an embedding 
\begin{equation}
H_M^\vee\colon (M_\sigma/M_1)^\vee\to \fa_M^*.
\end{equation}
For $m\in M$, let $b_m$ be the element of $\C[\fX_\nr(M)]$ defined by $b_m(\chi):=\chi(m)$ for any $\chi\in\fX_\nr(M)$.
Let $h_\alpha$ be the unique generator of $M_\sigma/M_1$ such that $\val_F(\alpha(h_\alpha))>0$. We define $X_\alpha\in \C(\fX_\nr(M)/\fX_{\nr}(M,\sigma)$ by
\begin{equation} \label{eqn:Xalpha}
    X_{\alpha}(\chi):=\chi(h_{\alpha}),
\end{equation}
for $\chi\in \fX_{\mathrm{nr}}(M)/\fX_{\mathrm{nr}}(M,\sigma)$. 
\end{numberedparagraph}

\begin{numberedparagraph}
For a complex number $s$, let $\chi_{s}$ be the character defined by
\begin{equation} \label{eqn:chi_s}
   \chi_s(m):= |\det(m)|_F^s \quad \text{for any $m\in M$.}
\end{equation}
In particular, $\chi_{s}\in \fX_\nr(M)$, and we have 
\begin{equation} \label{eqn:Xalphaevaluated}
 X_\alpha(\chi_s)=|\det(h_\alpha)|_F^s.   
\end{equation}
Let $\tilde\alpha$ be the element of $\fa^*_M$ defined by $\tilde\alpha:=\langle\rho_P,\alpha^\vee\rangle ^{-1}\rho_P$, where $\rho_P$ is half the sum of the roots of $A_M$ in $\mathrm{Lie}\,N$, with $P=MN$.
Then $s\tilde\alpha\in\fa_M^*\otimes_\R\C$.

We recall the description of the Plancherel measure from \cite{Silberger-79} (see also \cite{Solleveld-Hecke-algebra-2} or \cite{Heiermann-intertwining-operators-Hecke-algebras} for the notations used here): for $\alpha\in\Sigma_{\Oo,\mu}$, where $\Sigma_{\Oo,\mu}$ is the root system defined in \eqref{eqn:Root System}, there exist $q_{\alpha},q_{\alpha^*}\in \R_{\geq 1}, c'_{s_{\alpha}}\in\R_{>0}$ for $\alpha\in\Sigma_{\Oo,\mu}$, such that 
\begin{equation}\label{Silberger-Plancherel-formula}
    \mu^{M_\alpha}(\sigma\otimes\cdot)=c_{s_{\alpha}}'\,\frac{(1-X_{\alpha})(1-X_{\alpha}^{-1})}{(1-q_{\alpha}^{-1}X_{\alpha})(1-q_{\alpha}^{-1}X_{\alpha}^{-1})}\cdot\frac{(1+X_{\alpha})(1+X_{\alpha}^{-1})}{(1+q_{\alpha^*}^{-1}X_{\alpha})(1+q_{\alpha^*}^{-1}X_{\alpha}^{-1})}.
\end{equation}
\end{numberedparagraph}

\begin{numberedparagraph}
For $\alpha\in\Sigma_{\Oo,\mu}$, by \cite[Proposition~3.1]{Solleveld-endomorphism-algebra} there is a unique $\alpha^\sharp\in (M_\sigma/M_1)^\vee$ such that $H_M^\vee(\alpha^\sharp)\in\R\alpha$ and 
$\langle h_\alpha,\alpha^\sharp\rangle=2$. We set
\[\Sigma_\Oo:=\left\{\alpha^\sharp\,:\,\alpha\in\Sigma_{\Oo,\mu}\right\}
\quad\text{and}\quad
\Sigma_\Oo^\vee :=\left\{\alpha^\sharp\,:\,h_\alpha\in\Sigma_{\Oo,\mu}\right\}.\] 
The quadruple $(\Sigma_{\Oo}^{\vee},M_{\sigma}/M_1,\Sigma_{\Oo}, (M_{\sigma}/M_1)^{\vee})$ is a root datum with Weyl group $W_\Oo$. It has a natural action of the group $W(M,\Oo)$, and $R(\Oo)$ is the stabilizer
of its basis determined by $P$ (see \textit{loc.cit.}).
We endow this based root datum with the parameter $q_F$ and the labels
\begin{equation} \label{eqn:labels}
\lambda(\alpha):=\log(q_\alpha q_{\alpha^*})/\log(q_F)\quad\text{and}\quad
\lambda^*(\alpha):=\log(q_\alpha q_{\alpha^*}^{-1})/\log(q_F).
\end{equation}
To the above data we associate the affine Hecke algebra 
\begin{equation} \label{eqn:affineHA}
\mcH_\aff^\fs(\lambda,\lambda^*,q_F):=    \mcH_{\mathrm{aff}}(\Sigma_{\Oo}^{\vee},M_{\sigma}/M_1,\Sigma_{\Oo}, (M_{\sigma}/M_1)^{\vee}, \lambda,\lambda^*, q_F).
\end{equation}
It is defined as the vector space $\C[W_\Oo]\otimes_\C\C[M_{\sigma}/M_1]$
with the following multiplication rules:
\begin{itemize}
    \item $\C[W_\Oo]=\text{span}\{T_w: w\in W_\Oo\}$ is embedded as $\mcH(W_\Oo,q_F^\lambda)$, the \textit{Iwahori-Hecke algebra of $W_\Oo$}, i.e.
    \begin{equation} \label{eqn:IWH}
    \begin{matrix}
    T_wT_v=T_{wv}\quad\text{if $\ell(w)+\ell(v)=\ell(ww)$},\cr
    (T_{s_\alpha}+1)(T_{s_\alpha}-q_F^{\lambda(\alpha)})=0\quad\text{if $\alpha\in\Delta_{\Oo,\mu}$,}
    \end{matrix}
    \end{equation}
    where $\ell(w)$ is the word length of $w$\footnote{i.e., $\ell(w)$ is the smallest integer $\ell\ge 0$ such that $w$ is a product of $\ell$ generators $s_\alpha$.};
    \item $\C[M_{\sigma}/M_1]\simeq\C[\Oo]$ is embedded as a subalgebra;
    \item for $\alpha\in\Delta_{\Oo,\mu}$ and $x\in M_{\sigma}/M_1$ (corresponding to $\theta_x\in \C[M_{\sigma}/M_1]$):
    \[\theta_xT_{s_\alpha}-T_{s_\alpha}\theta_{s_\alpha(x)}=
    \left(q^{\lambda(\alpha)}_F-1+X_\alpha^{-1}(q^{\frac{\lambda(\alpha)+\lambda^*(\alpha)}{2}}-q^{\frac{\lambda(\alpha)-\lambda^*(\alpha)}{2}})\right)\frac{\theta_x-\theta_{s_\alpha(x)}}{1-X_\alpha^{-2}}. \]
\end{itemize}
\end{numberedparagraph}

\begin{numberedparagraph}
We set \begin{equation} \label{eqn:affine W}
W_\aff^\fs:=W_\Oo\ltimes\Z\Sigma_{\Oo}^{\vee}.
\end{equation}

From now on we assume that the parabolic subgroup $P$ is maximal. Then we have $M_\alpha=G$, and $W(M)$ is either trivial or of order $2$. 

\begin{remark} \label{WMO-at-most-order-2}
\begin{enumerate}
\item 
The groups $W(M,\Oo)$, $W_\Oo$, and $R(\Oo)$ are either trivial or of order $2$. 
In particular,  $\Sigma_{\Oo,\mu}$ is either empty or $\{\alpha,-\alpha\}$.
\item
For $G=\rG_2$, if $\sigma\not\simeq\sigma^{\vee}$, then $W(M,\Oo)=1$. It suffices to only check the case where $\sigma\simeq\sigma^{\vee}$.

In general, if $W(M,\Oo)=1$, then the parabolically induced representation is irreducible, so we do not need to work with the case. In the case of $\rG_2$, the condition $W(M,\Oo)\ne 1$ happens to be characterized by the condition that $\sigma$ is self-dual. See \cite{Shahidi-third-symmetric-power} for more details.
\end{enumerate}
\end{remark}
\end{numberedparagraph}

\begin{numberedparagraph}
If $\Sigma_{\Oo,\mu}\neq\emptyset$, then $W(M)\ne\{1\}$ and the group $W_\Oo$ is generated by the unique non-trivial element of $W(M)$, say $s_M$. Then we have $W_\Oo=W(M,\Oo)=W(M)$. In particular, if $\Sigma_{\Oo,\mu}\neq\emptyset$, we have $R(\Oo)=\{1\}$. 

The condition $\Sigma_{\Oo,\mu}=\emptyset$ is equivalent to the following
\begin{equation} \label{eqn:mu_nonzeo}
\mu^G(\chi\otimes\sigma)\ne 0\quad \text{ for any $\chi\in X_{\nr}(M)$}.
\end{equation}

We recall the following Harish-Chandra theorem. 
\begin{theorem}
\label{Harish-Chandra} {\rm (Harish-Chandra)}
\cite[5.4.2.2 and 5.4.2.3]{Silberger-79}  Let $M$ be a Levi subgroup of a maximal parabolic subgroup of $G$ and let $\sigma$ be a supercuspidal irreducible representation of $M$.\\
{\rm (a)} If $\mu^G(\sigma)=0$, then $W(M)=\{1,s_M\}\ne\{1\}$,  
and $s_M\sigma\simeq \sigma$.\\
{\rm (b)} Suppose $W(M)\ne\{1\}$. Then $\mu^G(\sigma)\neq 0$ if and only if the representation $i_{P}^{G}(\sigma)$ is reducible. In this case, the representation $i_{P}^{G}(\sigma)$ is the direct sum of two non-isomorphic irreducible representations.
\end{theorem}

\begin{Coro} \label{cor:HC}
Suppose $W(M)\ne\{1\}$. Then $W_\Oo=\{1\}$ if and only if, for any $\chi\in X_{\nr}(M)$, the representation $i_{P}^{G}(\chi\otimes\sigma)$ is reducible.
\end{Coro}
\end{numberedparagraph}

\subsubsection{\textbf{Some background on $\rG_2$}}
In the case where $G$ is the split $\rG_2$, we obtain more precise results than in previous sections.  Let $T$ be a maximal split torus in $G$. Let $R$ be the set of roots of $G$ with respect to $T$. Let
$(\epsilon_1,\epsilon_2,\epsilon_3)$ be the canonical
basis of $\R^3$, equipped with the scalar product $(\;|\;)$ for which this basis is orthonormal. Then
$\left\{\alpha:=\varepsilon_1-\varepsilon_2,
\beta:=-2\varepsilon_1+\varepsilon_2+\varepsilon_3\right\}$ defines a basis of $R$, and
\begin{equation}
R^+=\left\{\alpha,\beta,\alpha+\beta,2\alpha+\beta,3\alpha+\beta,
3\alpha+2\beta\right\}
\end{equation} 
is a subset of positive roots in $R$.
We have
\begin{equation} \label{products}
(\alpha|\alpha)=2,\;\;\; (\beta|\beta)=6\;\; \text{ and
}\;\;(\alpha|\beta)=-3.
\end{equation}
Hence $\alpha$ is a short root, while $\beta$ is a long root.
\begin{numberedparagraph}
As in \cite{Muic-unitary-dual-G2}, we fix an isomorphism:
\begin{align}\label{equ:etaa}
\eta_\alpha\colon T&\xlongrightarrow{\sim} F^{\times} \times F^{\times}\\
 t &\longmapsto ( (2\alpha+\beta)(t), (\alpha+\beta)(t) ).
\end{align}
Under this identification we have
\begin{equation} \label{eqn:eval}
\alpha^\vee(a)=\eta_\alpha^{-1}(a,a^{-1})\quad\text{and}\quad
\beta^\vee(a)=\eta_\alpha^{-1}(1,a)\quad\text{for any $a\in F^\times$.}
\end{equation}
Let $G^\vee$ be the dual group of $G$ over $\C$, obtained via an identification of the roots of $G^\vee$ with the coroots of $\rG_2$ and vice versa. Then $G^\vee$ is a complex reductive group of type $\rG_2$, with simple roots $\alpha^\vee$ and $\beta^\vee$. Note  that $\alpha^\vee$ (resp.~$\beta^\vee$) is the long (resp.~short) root of $G^\vee$. Consider the torus $T^\vee$ dual to $T$. Then $T^\vee$ is a maximal torus of $G^\vee$. 
We fix an isomorphism:
\begin{align} \label{etabC}
\eta_{\beta^\vee}\colon T&\xlongrightarrow{\sim}  \C^{\times} \times \C^{\times}\\
 t &\longmapsto ((\alpha^\vee+2\beta^\vee)(t), (\alpha^\vee+\beta^\vee)(t)).
\end{align}
We have 
\begin{equation} \label{eqn:evaldual}
\alpha^\vee(a)=\eta_{\beta^\vee}^{-1}(1,a)\quad\text{and}\quad
\beta^\vee(a)=\eta_{\beta^\vee}^{-1}(a,a^{-1})\quad\text{for any $a\in F^\times$.}
\end{equation}
\end{numberedparagraph}

\begin{numberedparagraph}
For each root $\gamma\in R(G)$, we fix root group homomorphisms $x_\gamma\colon F\to G$ and
$\Z$-homomorphisms $\zeta_\gamma\colon\SL_2(F)\to G$ 
as in \cite[(6.1.3)~(b)]{Bruhat-Tits-I}.
We have
\begin{equation}
x_\gamma(u)=\zeta_\gamma\left(\begin{matrix}
1&u\cr
0&1\end{matrix}\right),\;\;\;\;
x_{-\gamma}(u)=\zeta_{-\gamma}\left(\begin{matrix}
1&0\cr
u&1\end{matrix}\right)
\;\;\text{
and }\;\;
\gamma^\vee(t)=\zeta_\gamma\left(\begin{matrix}
t&0\cr
0&t^{-1}\end{matrix}\right).
\end{equation}
For $\gamma\in\{\alpha,\beta\}$, let $P_\gamma$ be the maximal standard parabolic subgroup of $G$ generated by $\gamma$. Let
$M_\gamma$ be the centralizer of the image of $(\gamma')^\vee$ in $G$, where $\gamma'$ is
the unique positive root orthogonal to $\gamma$, i.e.
\begin{equation}
\gamma':=\begin{cases}
3\alpha+\beta&\text{ if $\gamma=\alpha$,}\cr
3\alpha+2\beta&\text{ if $\gamma=\beta$.}
\end{cases}
\end{equation}
Then $M_\gamma$ is a Levi factor for $P_\gamma$. Moreover, $M_\alpha$ and $M_\beta$ are representatives of the two conjugacy classes of maximal Levi subgroups of $G$.

\noindent We extend $\zeta_\gamma\colon\SL_2(F)\to M_\gamma$ to an isomorphism
$\zeta_\gamma\colon\GL_2(F)\to M_\gamma$ by
\begin{equation}
\zeta_\gamma\left(\left(\begin{matrix}t&0\cr 0&1\end{matrix}\right)
\right):=\zeta_{\gamma'}\left(\left(\begin{matrix}t&0\cr
0&t^{-1}\end{matrix}\right)\right),
\;\;\;\;\text{for $t\in\ F^\times$.}
\end{equation}
Then the restriction of $\zeta_\gamma^{-1}$ to $T$ coincides
with the isomorphism
\begin{equation}
\eta_\gamma\colon T\to F^\times\times F^\times,
\end{equation}
where $\eta_\alpha$ has been defined in~(\ref{equ:etaa}), and
\begin{equation}
\eta_\beta\colon t\mapsto((\alpha+\beta)(t),\alpha(t)).
\end{equation}
\end{numberedparagraph}

\subsection{Explicit Hecke algebra parameters
}
\subsubsection{\textbf{The long root case}}
Let $\psi$ be a fixed nontrivial additive character of $F$, and $\overline{\psi}$ be the dual of $\psi$.
Assume for now the Levi factor $M$ of $P=MN$ is generated by the long root of $G$.  Let $\sigma$ be an irreducible unitary supercuspidal representation of $M$. We denote by $\omega:=\omega_\sigma$ the central character of $\sigma$. 
Let $L/F$ be a quadratic extension. Let $\chi$ be a character of $L^\times$. Let $\chi'$ be the conjugate of $\chi$, i.e. $\chi'(a)=\chi(\overline{a})$. Let $\Pi(\sigma)$ denote the Gelbart-Jacquet lift of $\sigma$ as defined in \cite{Gelbart-Jacquet-lift}. Our notations follow 
\cite{Shahidi-third-symmetric-power}. 
The Plancherel measure $\mu(s\widetilde{\alpha},\sigma)$ has the following four possibilities (\cite{AEFKY}). 

\begin{numberedparagraph}\label{subsub: long root I}\textbf{Case I.}~If $\omega$ is unramified, and if  $\sigma=\sigma(\tau)$ with
$\tau=\Ind_{W_L}^{W_F} \chi$,
with $\chi^2\chi'$ unramified, then
\begin{align}\label{long-root-case-I-Plancherel}
\mu(s\widetilde{\alpha},\sigma)=\gamma(G/P)^2q_F^{n(\omega) +n(\sigma \times \Pi(\sigma))-n(\sigma)}&\dfrac{(1-\omega(\varpi)q_F^{-2s})(1-\omega^{-1}(\varpi)q_F^{2s})}{(1-\omega^{-1}(\varpi)q_F^{-1+2s})(1-\omega(\varpi)q_F^{-1-2s})} \\
&\cdot \dfrac{(1-\chi^2\chi'^{-1}(\varpi_L)q_L^{-s})(1-\chi^{-2}\chi'(\varpi_L)q_L^s)}{(1-\chi^2\chi'^{-1}(\varpi_L)q_L^{-1-s})(1-\chi^{-2}\chi'(\varpi_L)q_L^{-1+s})}
\end{align}
Comparing \eqref{long-root-case-I-Plancherel} with the Plancherel formula in  \eqref{Silberger-Plancherel-formula}, we have
\begin{equation}
\begin{cases}
    X_{\alpha}(s)&=\omega(\varpi_F)q_F^{-2s}\\
    X_{\alpha}(s)&=-\chi^2\chi'^{-1}(\varpi_L)q_L^{-s},
\end{cases}
\end{equation}
which implies that 
\begin{equation}\label{Plancherel-equation-long-root-case-I}
\omega(\varpi_F)q_F^{-2s}+\chi^2\chi'^{-1}(\varpi_L)q_L^{-s}=0.
\end{equation}
Since $q_L=q^{f(L/F)}$, \eqref{Plancherel-equation-long-root-case-I} only has a solution when $f(L/F)=2$ and 
\begin{equation}
\omega(\varpi_F)+\chi^2\chi'^{-1}(\varpi_L)=0,
\end{equation}
which is satisfied in our case. In particular, we have
\begin{equation}\label{q-alpha-long-root-case1}
q_{\alpha}=q_F, \quad q_{\alpha^*}=q_L=q_F^{f(L/F)}.
\end{equation}
Therefore we have 
\begin{equation}
    \begin{matrix}\lambda(\alpha)=\log(q_{\alpha}q_{\alpha^*})/\log(q_F)=1+f(L/F),\\
    \lambda^*(\alpha)=|\log(q_{\alpha}q_{\alpha^*}^{-1})/\log(q_F)|=|1-f(L/F)|.\end{matrix}
\end{equation}
Hence the parameters for the extended affine Hecke algebra  
in this case are $q_F^{1+f(L/F)}$ and $q_F^{|1-f(L/F)|}$.
\end{numberedparagraph}

\begin{numberedparagraph}\label{subsub: long root II}\textbf{Case II.}~If $\omega_\sigma$ is ramified and $\sigma=\sigma(\tau)$
with $\tau=\Ind_{W_L}^{W_F} \chi$,  and $\chi^2\chi'$  unramified,
\begin{equation}\label{long-root-case-two}
\mu(s\widetilde{\alpha},\sigma)=\gamma(G/P)^2q_F^{n(\sigma \times \Pi(\sigma))-n(\sigma)}\dfrac{(1-\chi^2\chi'^{-1}(\varpi_L)q_L^{-s})(1-\chi^{-2}\chi'(\varpi_L)q_L^s)}{(1-\chi^2\chi'^{-1}(\varpi_L)q_L^{-1-s})(1-\chi^{-2}\chi'(\varpi_L)q_L^{-1+s})}
\end{equation}
We compare (\ref{long-root-case-two}) to the Plancherel formula (\ref{Silberger-Plancherel-formula}) and obtain
\begin{equation}\label{q-alpha-long-root-case2}
q_{\alpha^*}=1, \quad q_{\alpha}=q_L=q_F^{f(L/F)}
\end{equation}
Recall the definition of $X_{\alpha}$ as 
\begin{equation}
    X_{\alpha}(\chi):=\chi(h_{\alpha}^{\vee})
\end{equation}
where $\chi\in \mathfrak{X}^{\mathrm{nr}}(M)/\mathfrak{X}^{\mathrm{nr}}(M,\sigma)$. Since the map $\psi_s\colon m\mapsto |\det(m)|_F^s$ is an unramified character of $M$, we have 
\begin{equation}
X_{\alpha}(\psi_s)=(\chi^2\chi'^{-1})(\varpi_L)q_L^{-s}
\end{equation}
Recall from~(\ref{eqn:labels}) that
\begin{equation}\label{lambda-alpha}
    q_F^{\lambda(\alpha)}=q_{\alpha}q_{\alpha^*}\in\R_{>1}.
\end{equation}
Thus by \eqref{q-alpha-long-root-case2}, we have $q_F^{\lambda(\alpha)}=q_L=q_F^{f(L/F)}$, where $f(L/F)$ is the residue degree and is thus $1$ if $L/F$ is ramified, and $2$ if $L/F$ is unramified. In particular
\begin{equation}\label{group-side-lambda-1}
    \lambda(\alpha)=f(L/F), \quad\lambda^*(\alpha)=f(L/F).
\end{equation}
Note that for $w\in W(M,\Oo)$, one may check that
\begin{equation}\label{w-X-alpha=X-w-alpha-2}
    w(X_{\alpha})=X_{w(\alpha)}
\end{equation}
Since $w(\alpha)=\alpha$ for $w\in W(M,\Oo)$ when $G=\rG_2$, (\ref{w-X-alpha=X-w-alpha-2}) is simply $w(X_{\alpha})=X_{\alpha}$. On the other hand, by \cite[Prop 1.1]{Solleveld-Hecke-algebra-2} we have
\[wX_{\alpha}(\chi)=w(X_{\alpha}(\chi))=w(\chi(h_{\alpha}^{\vee}))=\chi(w(h_{\alpha}^{\vee}))=\chi(h_{w(\alpha)}^{\vee})=\chi(h_{\alpha}^{\vee})=X_{\alpha}(\chi)\]
Thus $wX_{\alpha}=X_{\alpha}=X_{w(\alpha)}$. 
This reduces to check, in the long root case, that 
\begin{equation}
    s_{2\alpha+\beta}(\chi^2\chi'^{-1}(\varpi_L)q_L^{-s})=\chi^2\chi'^{-1}(\varpi_L)q_L^{-s}
\end{equation}
Since $\Sigma_{\Oo}^{\vee}=\{1, 2\alpha+\beta\}$ in the long root $M=M^{\beta}$ case, we have
\begin{equation}\label{Ws-Galois-side}
    W(\Sigma_{\Oo}^{\vee})=\{1,s_{2\alpha+\beta}\}.
\end{equation}
Thus the Iwahori-Hecke algebra of $W(\Sigma_{\Oo}^{\vee})$, as defined in \eqref{eqn:IWH}, is given by
\begin{equation}
    \mathcal{H}(W(\Sigma_{\Oo}^{\vee}),q_F^{\lambda})=\begin{cases}\mathcal{H}(\{1,s_{2\alpha+\beta}\},q_F),&L/F\text{ is ramified}\\
    \mathcal{H}(\{1,s_{2\alpha+\beta}\},q_F^2),&L/F\text{ is unramified}
    \end{cases}
\end{equation}
Therefore, the affine Hecke algebra in this case is given by
\begin{equation}
    \mathcal{H}_{\mathrm{aff}}(M^{\beta})= \mathcal{H}(\{1,s_{2\alpha+\beta}\},q_F^{f(L/F)})\ltimes\C[\Oo]
\end{equation}
\end{numberedparagraph}

\begin{numberedparagraph}\label{subsub: long root III}\textbf{Case III.}~If $\omega$ is unramified and 
$\sigma \neq \sigma(\tau)$ or 
$\chi^2\chi'$ is ramified, 
\begin{equation}\label{long-root-case-three}
\mu(s\widetilde{\alpha},\sigma)=\gamma(G/P)^2q_F^{n(\omega) +n(\sigma \times \Pi(\sigma))-n(\sigma)}\dfrac{(1-\omega(\varpi)q_F^{-2s})(1-\omega^{-1}(\varpi)q_F^{2s})}{(1-\omega^{-1}(\varpi)q_F^{-1+2s})(1-\omega(\varpi)q_F^{-1-2s})}
\end{equation}
In this case, we have 
\begin{equation}
\begin{matrix}
    X_{\alpha}(s)=\omega(\varpi)q_F^{-2s}\\
    q_{\alpha^*}=1, \quad q_{\alpha}=q_F
    \end{matrix}
\end{equation}
Thus $\lambda(\alpha)=1$ and $\lambda^*(\alpha)=1$. The parameters in this case are simply $q_F$. 
\end{numberedparagraph}

\begin{numberedparagraph}\label{subsub: long root IV}\textbf{Case IV.}~If $\omega$ ramified and 
$\sigma \neq \sigma(\tau)$ or 
$\chi^2\chi'$ is ramified,
\begin{equation}\label{long-root-case-four}
\mu(s\widetilde{\alpha},\sigma)=\gamma(G/P)^2q_F^{n(\sigma \times \Pi(\sigma)))-n(\sigma)}
\end{equation}
In this case, we have 
\begin{equation}
    q_{\alpha}=1,\quad q_{\alpha^*}=1
\end{equation}
Thus $\lambda(\alpha)=0$ and $\lambda^*(\alpha)=0$. Thus the parameters in this case are trivial.
\end{numberedparagraph}

\subsubsection{\textbf{The short root case}} \label{subsec:short root}
Now we give the explicit computation in the short root case. Assume the Levi factor $M$ of $P=MN$ is generated by the short root of split $\rG_2$. Let $\sigma$ be an irreducible unitary supercuspidal representation of $M$. Let $\omega=\omega_\sigma$ be its central character. Then by \cite[Proposition 6.2]{Shahidi-Plancherel-Langlands} the Plancherel measure $\mu(s\widetilde{\alpha},\sigma)$ is given by the formula 
\begin{equation}\label{short-root-formula}
\mu(s\widetilde{\alpha},\sigma)=\begin{cases}\gamma(G/P)^2q_F^{n(\sigma)+n(\sigma\otimes\omega)}\frac{(1-\omega(\varpi)q_F^{-2s})(1-\omega(\varpi)^{-1}q_F^{2s})}{(1-\omega(\varpi)q^{-1-2s})(1-\omega(\varpi)^{-1}q_F^{-1+2s})} &\text{if $\omega$ is unramified}\\
\gamma(G/P)^2q_F^{n(\sigma)+n(\omega)+n(\sigma\otimes\omega)}&\text{otherwise}\end{cases}
\end{equation}
Here $n(\sigma)$, $n(\omega)$ and $n(\sigma\otimes\omega)$ are the corresponding conductors.

\begin{numberedparagraph}\label{subsub: short root unramified} \textbf{Case I.}~
If $\omega$ is unramified, 
\begin{equation}\label{unramified-short-root}
    \mu(s\widetilde{\alpha},\sigma)=\gamma(G/P)^2q_F^{n(\sigma)+n(\sigma\otimes\omega)}\frac{(1-\omega(\varpi)q_F^{-2s})(1-\omega(\varpi)^{-1}q_F^{2s})}{(1-\omega(\varpi)q_F^{-1-2s})(1-\omega(\varpi)^{-1}q_F^{-1+2s})} 
\end{equation}
Comparing \eqref{unramified-short-root} with \eqref{Silberger-Plancherel-formula} implies that 
\begin{equation}\label{q-unramified-short-root}
    q_{\alpha}=q_F,\quad q_{\alpha^*}=1
\end{equation}
Since $\chi_s$ is an unramified character of $M$, we have 
\begin{equation}
    X_{\alpha}(\chi_s)=\omega(\varpi)q_F^{-2s}.
\end{equation}
Recall from~(\ref{eqn:labels}) that 
    $q_F^{\lambda(\alpha)}=q_{\alpha}q_{\alpha^*}\in\R_{>1}$. 
Thus by \eqref{q-unramified-short-root},
we have $q_F^{\lambda(\alpha)}=q_F$ and thus $\lambda(\alpha)=1$ and $\lambda^*(\alpha)=1$. 
Note that for $w\in W(M,\Oo)$, one may check that
\begin{equation}\label{w-X-alpha=X-w-alpha}
    w(X_{\alpha})=X_{w(\alpha)}.
\end{equation}
Since $w(\alpha)=\alpha$ for $w\in W(M,\Oo)$ for $G=\rG_2$, (\ref{w-X-alpha=X-w-alpha}) is simply $w(X_{\alpha})=X_{\alpha}$. On the other hand, by \cite[Prop 1.1]{Solleveld-Hecke-algebra-2} we have
\begin{equation}
wX_{\alpha}(\chi)=w(X_{\alpha}(\chi))=w(\chi(h_{\alpha}))=\chi(w(h_{\alpha}))=\chi(h_{w(\alpha)})=\chi(h_{\alpha})=X_{\alpha}(\chi).
\end{equation}
Thus $wX_{\alpha}=X_{\alpha}=X_{w(\alpha)}$. 
Since $\Sigma_{\Oo}^{\vee}=\{3\alpha+2\beta\}$ in the short root $M=M_\alpha$ case (see \cite[p.389]{Shahidi-Plancherel-Langlands}), we have
\begin{equation}\label{Ws-group-side}
    W(\Sigma_{\Oo}^{\vee})=\{1,s_{3\alpha+2\beta}\}
\end{equation}
Therefore we have the affine Hecke algebra
\begin{equation}
    \mathcal{H}_{\mathrm{aff}}(M_{\alpha})=\mathcal{H}(\{1,s_{3\alpha+2\beta}\},q_F)\ltimes\C[\Oo].
\end{equation}
\end{numberedparagraph}

\begin{numberedparagraph}
\label{subsub: short root ramified} 
\textbf{Case II.}~In the other case, 
\begin{equation}\label{short-root-ramified}
    \mu(s\widetilde{\alpha},\sigma)=\gamma(G/P)^2q^{n(\sigma)+n(\omega)+n(\sigma\otimes\omega)}
\end{equation}
Comparing \eqref{short-root-ramified} and \eqref{Silberger-Plancherel-formula} gives us
\begin{equation}
    q_{\alpha}=1,\quad q_{\alpha^*}=1.
\end{equation}
Therefore we have $q_F^{\lambda}=1$ and thus $\lambda(\alpha)=0$ and $\lambda^*(\alpha)=0$. Therefore the affine Hecke algebra in this case is given by
\begin{equation}
    \mathcal{H}_{\mathrm{aff}}(M_{\alpha})=\mathcal{H}(\{1,s_{3\alpha+2\beta}\},1)\ltimes\C[\Oo].
\end{equation}
\begin{remark}
The computations of Hecke algebras with explicit parameters in this section will be collected into tables in $\mathsection$\ref{subsec:IA}. 
\end{remark}
\end{numberedparagraph}

\subsection{Intertwining algebras} \label{subsec:Kim-Yu}
\begin{numberedparagraph}
For $b \in F^\times/F^{\times 2}$, let $\rU_b(1,1)$ be the quasi-split unitary group, and $\rU_b(2)$ the compact unitary group in two variables in $F (\sqrt{b})$. We write $F^\times/F^{\times 2}= \{1,\varepsilon,\varpi,\varepsilon\varpi\}$, and the possible unitary groups in $2$ variables are:
\[ \rU_{\varepsilon}(1,1),\ \rU_{\varepsilon}(2),\ \rU_{\varpi}(1,1),\  \rU_{\varpi}(2),\ \rU_{\varepsilon\varpi}(1,1),\ \rU_{\varepsilon\varpi}(2).\]
The group $\rU_{\epsilon}(1,1)$ is an unramified group. The group $\rU_{\varpi'}(1,1)$ is ramified, where $\varpi'\in\{\varpi,\varepsilon\varpi\}$.
\end{numberedparagraph}

\begin{numberedparagraph}
We now classify the twisted Levi sequences in $\rG_2$ (up to conjugacy) for $\bM=\bM_\gamma$ with $\gamma\in\{\alpha,\beta\}$:
\begin{enumerate}
    \item[(1)]
    Essentially depth zero case: 
    If $\rho_{\Sigma_M}$ is an essentially depth-zero supercuspidal type on $M$, then $\Sigma_M$ is of the form $(M,y,\phi,r,\rho_M)$ (hence in particular $M^0=M$), where $K_{\Sigma_M}=M_{y,0}\simeq\GL_2(\fo_F)$ is a maximal compact subgroup of $\GL_2(F)$ and $r=\depth(\rho_{\Sigma_M})$ is an integer. If $r=0$, we may assume that $\phi=1$ without loss of generality.
    \begin{itemize}
\item[(a)] $\vec \bG=(\bG)$ (here $\bG^0=\bG$, it is a depth zero case: $r=0$),
\item[(b)] $\vec \bG=(\bM^0,\bG)$ (here $\bG^0=\bM=\bM^0$ and $r\ne 0$).
    \end{itemize}
    \item[(2)]
    Positive depth cases \cite{AEFKY}: 
\begin{itemize}
\item[(a)] $\vec \bG=(\rU_{\varepsilon}(1,1),\bG)$, 
\item[(b)] $\vec \bG=(\rU_{\varpi'}(1,1),\bG)$, with $\varpi'\in\{\varpi,\varepsilon\varpi\}$,
\item[(c)] $\vec \bG=(\bM^0,\bG)$,
\item[(d)] $\vec \bG=(\bM^0,\bM,\bG)$,  
\end{itemize}
where $\bM^0$ is a torus in $\bG^0$.
\end{enumerate}
When $\bM=\bM_\gamma$, we have three possibilities for $\bM^0$, denoted $T_{\gamma,\varepsilon}$, $T_{\gamma,\varpi}$ and $T_{\gamma,\varpi}$.
If $\phi_0$ has trivial restriction to $\rZ_M^\circ$, then it can be extended to to a character of $\rU_b$ and we use the same notation $\phi_0$ to denote the extended character.

Let $\bbG^0_y$ denote the reductive quotient of $G^0_y$. Let $\varpi'\in\{\varpi,\varepsilon\varpi\}$. We have
\begin{equation}
\bbG^0_y=\begin{cases}
\rU(1,1)&\text{if $G^0=\rU_{\varepsilon}(1,1)$,}\cr
\SO_2&\text{if $G^0=\rU_{\varpi'}(1,1)$.}\end{cases} 
\end{equation}

\begin{remark} \label{rem:central characters}
The central character $\omega_{\sigma}$ of $\sigma$ can be either ramified or unramified. It is unramified if and only if $\omega_{\sigma^0}$ is trivial. When  $\omega_{\sigma}$ is ramified, $\omega_{\sigma^0}$ is quadratic.
\end{remark}

\begin{lem} \label{lem:Ad-Mi}
We have 
\begin{equation}
W_G^\fs\simeq W_{G^0}^{\fs^0}.
\end{equation}
\end{lem}
\begin{proof} The representation $\sigma$ is regular and $p$ is good for $G$, i.e.~$p\ne 2,3$ and does not divide the order ($=1$) of the fundamental group of $\bG_\der$. Hence Lemma~\ref{lem:GG0} applies and gives the desired isomorphism. 
\end{proof}
\end{numberedparagraph}

\subsubsection{\textbf{The intertwining algebras of types attached to $G^0$}} \label{subsec:IAG0}
\begin{numberedparagraph}\label{subsubsec:G0M0} 
\textbf{The case \texorpdfstring{$G^0=M^0$}{G^0=M^0}.}~It occurs in both the essentially depth-zero case with $r\ne 0$  and in the positive depth cases.
We have two possibilities for $M^0$: either $M^0\simeq \GL_2(F)$ or $M^0$ is a torus. In both cases, the algebra $\mathcal{H}(G^0,\rho_{\mcD^0})=\mcH(M^0,\rho_{\mcD^0})$ is commutative by \cite[5.5,5.6]{Bushnell-Kutzko}.
\end{numberedparagraph}

\begin{numberedparagraph}\label{subsub: Ue}\textbf{The case $G^0=\rU_{\varepsilon}(1,1)$.}~
If $W_{G^0}^{\fs^0}=\{1\}$, then $\mathcal{H}(G^0,\rho_{\mcD^0})$ is commutative, as seen in Remark~\ref{rem: trivial commutative case}. 
From now on we suppose that $W_{G^0}^{\fs^0}\ne\{1\}$.
Let $a\mapsto \overline{a}$ be the non-trivial element of $\Gal(L/F)$. Set
\begin{equation}
w^0:=\left(\begin{matrix}0&1\cr
1&0\end{matrix}\right),\quad
w_1:=\left(\begin{matrix}0&\overline\varpi_L^{-1}\cr
\varpi_L&0\end{matrix}\right),\quad\textbf{and}
\quad\fP:=\left(\begin{matrix} \fo_L^\times&\fo_L\cr\fp_L&\fo_L^\times\end{matrix}\right)\cap G^0.
\end{equation}
Recall that $\widetilde{\rho_{\mcD^0}}$ denotes the contragredient representation of $\rho^0$.
By \cite[\S3.1]{Badea}, the  Iwahori-Matsumoto presentation of $\mcH(\rU_{\varepsilon}(1,1),\rho_{\mcD^0})$ is given by: $\mcH(\rU_{\varepsilon}(1,1),\rho_{\mcD^0})$ is the space  spanned by functions 
\begin{equation}
T_{w_i}\colon G^0 \to \End_G(V_{\widetilde{\rho_{\mcD^0}}}), \quad\text{for $i\in\{0,1\}$,}
\end{equation}
satisfying
\begin{equation} \label{eqn:AHecke}
T_{w_i}(pgp')=\widetilde{\rho_{\mcD^0}}(p)T_{w_i}(g)\widetilde{\rho_{\mcD^0}}(p'),\quad \text{where $p,p'\in \fP$ and $g\in G^0$.}\end{equation}
Here $T_{w_i}$ is supported on $\fP w_i\fP$, and satisfies the quadratic relation 
\begin{equation} \label{eqn:quadUe}
(T_{w_i}-q_F)(T_{w_i}+1)=0.
\end{equation}
One can then deduce the Bernstein presentation of $\mcH(\rU_{\varepsilon}(1,1),\rho_{\mcD^0})$ using \cite[\S3]{Lusztig-affine-Hecke-algebras-and-graded}. In particular, we have $q_F^{\lambda(\alpha)}=q_F$.
\end{numberedparagraph}

\begin{numberedparagraph}\label{subsub: Upi}\textbf{The case $G^0=\rU_{\varpi'}(1,1)$.}~
Let $\varpi'\in\{\varpi,\varepsilon\varpi\}$. Since $\rU_{\varpi'}(1,1)$ is ramified, by \cite[\S5.1.1]{Badea}, the algebra $\mcH(\rU_{\varpi'}(1,1),\rho_{\mcD^0})$ has trivial parameters 
with $R(\Oo^0)\neq 1$ and $W_{\Oo^0}=1$ if $\omega_\sigma|_{\fo^\times_F}\ne 1$; 
and the Hecke algebra has parameter $q_F$ otherwise, in which case $W_{\Oo^0}\neq 1$ and $R(\Oo^0)=1$.
\end{numberedparagraph}

\subsubsection{\textbf{The intertwining algebras of types attached to \texorpdfstring{$G$}{G}}} \label{subsec:IA}

\begin{numberedparagraph}\label{subsec:long essentially depth zero}\textbf{Long root essentially depth zero case.}\\
\noindent{{ (a) $r=0$, $\chi^3=1$ case and $\sigma=\sigma(\tau)$ for $\tau=\mathrm{Ind}_{W_L}^{W_F}\chi$:}} 

We  have $\rho_M$ self-dual, $\sigma$ and $\tau$ correspond via LLC for $\GL_2(F)$. Since  $\sigma$ has depth zero, $L/F$ is unramified (so $e(L/F)=1$ and $f(L/F)=2$). We have the following four cases:
\begin{itemize}
    \item 
The central character $\omega_{\sigma}=1$ and $\chi^2\chi'^{-1}$ unramified. This corresponds to Case \ref{subsub: long root I}, in which case the Plancherel formula has a zero, and the Hecke algebra is affine non-commutative, with parameters $q_F^3$ and $q_F$. 
We have $W_{\mathcal{O}}\neq 1$ and $R(\mathcal{O})=1$. Since $G=G^0$ in this case, $W_{\mathcal{O}}=W_{\mathcal{O}^0}$ and $R(\mathcal{O})=R(\mathcal{O}^0)$.
\item
The central character $\omega_{\sigma}\neq 1$ is ramified, and $\chi^2\chi'^{-1}$ is unramified. This corresponds to Case \ref{subsub: long root II}, in which case the Plancherel formula has a zero, and the Hecke algebra is affine non-commutative, with parameters $q_F^2$. We have $W_{\mathcal{O}}\neq 1$ and $R(\mathcal{O})=1$. Since $G=G^0$ in this case, $W_{\mathcal{O}}=W_{\mathcal{O}^0}$ and $R(\mathcal{O})=R(\mathcal{O}^0)$.
\item
The central character $\omega_{\sigma}=1$  and $\chi^2\chi'^{-1}$ ramified. This corresponds to Case \ref{subsub: long root III}, in which case the Plancherel formula has a zero, and the Hecke algebra is affine non-commutative, with parameter $q_F$. We have $W_{\mathcal{O}}\neq 1$ and $R(\mathcal{O})=1$. Since $G=G^0$ in this case, $W_{\mathcal{O}}=W_{\mathcal{O}^0}$ and $R(\mathcal{O})=R(\mathcal{O}^0)$.
\item
The central character $\omega_{\sigma}\neq 1$ is ramified, and $\chi^2\chi'^{-1}$ is ramified. This corresponds to Case \ref{subsub: long root IV}, in which case the Plancherel formula has no zero, and the Hecke algebra is affine commutative of the form $\mathbb{C}[R(\mathcal{O})]$ plus the translation part $\C[\mathcal{O}]$. We have  $W_{\mathcal{O}}=1$ (and we don't know what $R(\mathcal{O})$ is in this case). Since $G=G^0$ in this case, $W_{\mathcal{O}}=W_{\mathcal{O}^0}$ and $R(\mathcal{O})=R(\mathcal{O}^0)$.
\end{itemize}

\noindent{{ (b) $r=0$ and $\sigma\neq \sigma(\tau)$:}}
We have $\sigma=\sigma(\tau')$ where $\tau'=\mathrm{Ind}_{W_L}^{W_F}\zeta$ for $\zeta$ such that $\zeta^{-1}=\overline\zeta$ (the Galois conjugate). Since $\sigma$ is still depth zero, we still have $L/F$ unramified. 
\begin{itemize}
\item 
The central character $\omega_{\sigma}=1$. This corresponds to Case \ref{subsub: long root III}, in which case the Plancherel formula has a zero, and the Hecke algebra is affine non-commutative, with parameters $q_F$. We have $W_{\mathcal{O}}\neq \{1\}$ and $R(\mathcal{O})=\{1\}$. Since $G=G^0$ in this case, $W_{\mathcal{O}}=W_{\mathcal{O}^0}$ and $R(\mathcal{O})=R(\mathcal{O}^0)$.
\item
The central character $\omega_{\sigma}\neq 1$ ramified. This corresponds to Case \ref{subsub: long root IV}, in which case the Plancherel formula has no zero, and the Hecke algebra is affine commutative of the form $\mathbb{C}[R(\mathcal{O})]$ plus the translation part $\C[\mathcal{O}]$. we have $W_{\mathcal{O}}=\{1\}$ (and we don't know what $R(\mathcal{O})$ is in this case). Since $G=G^0$ in this case, $W_{\mathcal{O}}=W_{\mathcal{O}^0}$ and $R(\mathcal{O})=R(\mathcal{O}^0)$.
\end{itemize}

\noindent{{ (c) $r\neq 0$ essentially depth zero case:}} Recall from \S\ref{subsubsec:G0M0} that $G^0=M=M^0$.
Thus we have 
\[W_{G^0}^{\fs^0}\subset \Nor_{G^0}(M^0)/M^0=\Nor_{M}(M)/M=\{1\}.\]
By Lemma~\ref{lem:Ad-Mi}, we get $W_G^\fs=\{1\}$. In this case, $W(M,\Oo)=W(M^0,\Oo^0)=1$. Thus the algebras $\mcH(G,\rho)$ and $\mcH(G^0,\rho^0)$ are both of the form $\C[\Oo]$, and they are isomorphic.
\end{numberedparagraph}

\begin{numberedparagraph}\label{num: M long 0}
\textbf{Table for long root essentially depth zero cases.}\\

\adjustbox{scale=0.7,center}{
\begin{tabular}{|c|c|c|c|c|c|c|c|c|c|c|c|}
\hline
$r$ & $\mcD$ &  $\omega_{\sigma}$ &$\chi^2\chi'^{-1}$  & $R(\mathcal{O})$ & $R(\Oo^0)$ & $L/F$
& $\#X_{\nr}(M,\sigma)$ & $W_{\Oo}$ &$W_{\Oo^0}$ & $\mathcal{H}(G,\rho)$ & $\mathcal{H}(G^0,\rho^0)$\\ \hline\hline
\multirow{4}{*}{$r=0$}& \multirow{4}{*}{$((G,M),(y,\iota),(M_{y,0},\rho_M))$} & $=1$ &$\begin{array}{c} \text{unramified} \\ \chi \text{ cubic}\end{array}$
   & $=1$ & $=1$
& $\text{unramified}$ & $2$ 
& $\neq 1$ & $\neq 1$ & non-comm, $q_F^3$, $q_F$  &non-comm, $q_F^3,q_F$ 
\\ \cline{3-12}
&  & $\ne 1$ &$\begin{array}{c} \text{unramified}\\ \chi \text{ cubic}\end{array}$  & $=1$ & $= 1$  & unramified & $2$ & $\neq 1$ & $\neq 1$ & non-comm, $q_F^2$, $q_F^2$ &non-comm, $q_F^2$, $q_F^2$
\\ \cline{3-12}
&  &$=1$ &$\begin{array}{c}\text{ramified}\\ \chi \text{ cubic}\end{array}$ & $=1$  & $= 1$ & unramified
 & 2 &  $\neq 1$
& $\neq 1$  &non-comm, $q_F$, $q_F$ &non-comm, $q_F$, $q_F$
\\ \cline{3-12}
& &$\neq 1$ &$\begin{array}{c}\text{ ramified}\\ \chi \text{ cubic}\end{array}$ &$*$ &$*$ &unramified &2 &$=1$ &$=1$ & $\C[R(\mathcal{O})]\ltimes\C[\Oo]$ &$\C[R(\mathcal{O})]\ltimes\C[\Oo]$\\
\cline{3-12}
& &$=1$ &\multirow{2}{*}{$\begin{array}{c}\chi\text{ not cubic}\\ \text{N/A}\end{array}$} &$=1$ &$=1$ & unramified &2&$\neq 1$ &$\neq 1$ &non-comm,$q_F$,$q_F$ &non-comm,$q_F$,$q_F$\\
\cline{3-3}\cline{5-12}
&&$\neq 1$ &&$*$ &$*$ &unramified&2&$=1$ &$=1$&$\C[R(\mathcal{O})]\ltimes\C[\Oo]$ &$\C[R(\mathcal{O})]\ltimes\C[\Oo]$\\
\hline
\multirow{2}{*}{$r\neq 0$}
& \multirow{2}{*}{$(((M,G),M),(y,\iota),(r,0),(\phi,1),(M_{y,0},\rho_M))$}
& \multirow{2}{*}{$\ne 1$}
  &\multirow{2}{*}{N/A}   & \multirow{2}{*}{$=1$}
&\multirow{2}{*}{$=1$}  & \multirow{2}{*}{unramified}
& \multirow{2}{*}{2} & \multirow{2}{*}{$=1$} &\multirow{2}{*}{$=1$} &\multirow{2}{*}{$\C[\Oo]$}&\multirow{2}{*}{$\C[\Oo]$}
\\ 
& 
& 
  &  & 
&
&
& 
&  
& 
&
&
\\ \hline
\end{tabular}}
\smallbreak
\begin{center} {\sc Table} \ref{num: M long 0}.
\end{center}
\end{numberedparagraph}

\begin{numberedparagraph}\textbf{Long root positive depth case}\\
\noindent{(a)
$\rU_{\varpi'}(1,1)$ case: $\sigma=\sigma(\tau')\neq\sigma(\tau)$, where $\tau'$ is induction of some quadratic character.} (Note that the cubic character only occurs in depth zero, because we are assuming  $p\neq 2,3$. 
There are two possibilities, $\phi_0|_{\rZ_M^0}$ could be either trivial or non-trivial:
\begin{itemize}
    \item 
When $\phi_0|_{\rZ_M^0}=1$ unramified, since $\sigma=\sigma(\tau')\neq\sigma(\tau)$, this corresponds to Case \ref{subsub: long root III}, in which case the Plancherel formula has a zero, and the Hecke algebra is affine non-commutative, with parameters $q_F$. We have $W_{\mathcal{O}}\neq \{1\}$ and $R(\mathcal{O})=1$. By~\ref{subsub: Upi}, the Hecke algebra for $G^0$ also has $q_F$ parameter. Thus we have $W_{\mathcal{O}^0}\neq \{1\}$ and $R(\mathcal{O}^0)=1$. 
\item When $\phi_0|_{\rZ_M^0}=$ sign character ramified, since $\sigma=\sigma(\tau')\neq\sigma(\tau)$, this corresponds to Case \ref{subsub: long root IV}, in which case the Plancherel formula has no zero, and the Hecke algebra is of the form $\mathbb{C}[R(\mathcal{O})]\ltimes\C[\Oo]$. We have $W_{\mathcal{O}}=\{1\}$. By~\ref{subsub: Upi}, we have $W_{\mathcal{O}^0}=\{1\}$ and $R(\Oo^0)\neq 1$. Thus $R(\Oo)\cong R(\Oo^0)\neq 1$ by Lemma \ref{lem:GG0}.  
\end{itemize}

\noindent{{(b)
$\rU_{\epsilon}(1,1)$ case: $\sigma=\sigma(\tau')\neq\sigma(\tau)$, where $\tau'$ is the induction of some quadratic character.}} 
\begin{itemize}
\item
When $\phi_0|_{Z_M^0}=1$ unramified, since $\sigma=\sigma(\tau')\neq\sigma(\tau)$, this corresponds to \ref{subsub: long root III}, in which case the Plancherel formula has a zero, and the Hecke algebra for $G$ is affine non-commutative, with parameters $q$. We have $W_{\mathcal{O}}\neq \{1\}$ and $R(\mathcal{O})=\{1\}$. From~\ref{subsub: Ue}, we have $W_{\mathcal{O}^0}\ne \{1\}$ and $R(\mathcal{O}^0)=\{1\}$. Note that the cardinality of $\fX_\nr(M,\sigma(\tau'))$ is $2$ (see Remark~\ref{rem:size}).
\end{itemize}
\end{numberedparagraph}

\begin{numberedparagraph}\label{num: M long pos}\textbf{Table for long root positive depth cases.}~We summarize the above in the following table:

\adjustbox{scale=0.841,center}{
\begin{tabular}{|c|c|c|c|c|c|c|c|c|c|c|c|}
\hline
$M^0$ &  $\phi_0|_{\rZ_M^0}$ &$\phi_1$  &$\vec G$ & $R(\mathcal{O})$ & $R(\Oo^0)$ & $L/F$
& $\#X_{\nr}(M,\mcD)$ & $W_{\Oo}$ &$W_{\Oo^0}$ & $\mathcal{H}(G,\rho)$ & $\mathcal{H}(G^0,\rho^0)$\\ \hline\hline
\multirow{3}{*}{$T_{\beta,\varpi'}$}&$\begin{array}{c}=\text{ sign}\\ \text{character}\end{array}$  &$\ne 1$   &$(U_{\varpi'}(1,1),G)$  &$\neq 1$&$\neq 1$&ramified&1&$=1$&$=1$&$\mathbb{C}[R(\mathcal{O})]\ltimes\C[\Oo]$&$\mathbb{C}[R(\mathcal{O})]\ltimes\C[\Oo]$\\
\cline{2-12}
& $\begin{array}{c}\ne 1\\ \neq\text{sign character} \end{array}$ & $= 1$ & $(M^0,G)$  & $= 1$ &$= 1$  & ramified & $1$ & $=1$ & $=1$ &$\C[\Oo]$& $\C[\Oo^0]$
\\ \cline{2-12}
&both & $\neq 1$ & $\begin{array}{c}(M^0,M,G)\end{array}$ & $= 1$ & $= 1$ & ramified
 & 1 &$=1$
& $=1$  &$\C[\Oo]$ &$\C[\Oo^0]$
\\ \hline
\multirow{6}{*}{$T_{\beta,\varepsilon}$}
& \multirow{2}{*}{$=1$} &\multirow{2}{*}{$=1$}
& \multirow{2}{*}{$(U_\varepsilon(1,1),G)$}  & \multirow{2}{*}{$=1$}  & \multirow{2}{*}{$=1$}
& \multirow{2}{*}{unramified} & \multirow{2}{*}{2}
& \multirow{2}{*}{$\neq 1$} & \multirow{2}{*}{$\neq 1$} &\multirow{2}{*}{$\begin{array}{c}\text{non-comm.}\\ q_F, q_F\end{array}$} &\multirow{2}{*}{$\begin{array}{c}\text{non-comm.}\\q_F,q_F\end{array}$}
\\ 
& & & &  &  & & & && &
\\ \cline{2-12}
& \multirow{2}{*}{$\ne 1$ } &\multirow{2}{*}{$=1$}
& \multirow{2}{*}{$(M^0,G)$}  & \multirow{2}{*}{$=1$} & \multirow{2}{*}{$= 1$}
& \multirow{2}{*}{unramified}
& \multirow{2}{*}{2} & \multirow{2}{*}{$=1$} & \multirow{2}{*}{$=1$} &\multirow{2}{*}{$\C[\Oo]$} &\multirow{2}{*}{$\C[\Oo^0]$}
\\ 
& & & &  &  & & & && &
\\ \cline{2-12}
& \multirow{2}{*}{\ both} &\multirow{2}{*}{$\neq 1$}
&  \multirow{2}{*}{$\begin{array}{c}(M^0,M,G)\end{array}$} & \multirow{2}{*}{$=1$} & \multirow{2}{*}{$= 1$}
& \multirow{2}{*}{unramified} 
& \multirow{2}{*}{2}  & \multirow{2}{*}{$=1$} 
& \multirow{2}{*}{$=1$} &\multirow{2}{*}{$\C[\Oo]$} &\multirow{2}{*}{$\C[\Oo^0]$} 
\\ 
& & & &  &  & & & && &
\\ \hline
\end{tabular}}
\smallbreak
\begin{center}
{\sc Table} \ref{num: M long pos}.
\end{center}
\end{numberedparagraph}

\begin{numberedparagraph}\label{subsec:short essentially depth zero}\textbf{Short root essentially depth zero case.}\\
\noindent{{\bf (a) $r=0$,}} there are only two cases:
\begin{itemize}
    \item When $\rho_M|_{Z_M^{\circ}}=1$, this corresponds to the central character being unramified case, and in this case the Plancherel formula in \ref{subsub: short root unramified} has a zero. Thus thus $W_{\mathcal{O}}\neq 1$ and thus $R(\mathcal{O})=1$. In this case the Hecke algebra is non-commutative, and the $q$-parameter is just $q=q_F$. The case for $G^0$ again follows from \ref{subsub: Upi}. 
    \item When $\rho_M|_{Z_M^{\circ}}\neq 1$, this corresponds to the central character being ramified case, and in this case the Plancherel formula in \ref{subsub: short root ramified} has no zero, and thus $W_{\mathcal{O}}=\{1\}$. In this case the Hecke algebra is commutative, and the $q$-parameter is trivial.
\end{itemize}

\noindent{{ (b) $r\neq 0$ essentially depth zero case.}} 
The same argument as in \S\ref{subsec:long essentially depth zero}(c) applies. 
\end{numberedparagraph}

\begin{numberedparagraph}\label{num: M short 0}\textbf{Table for short root essentially depth zero cases.}\

    \adjustbox{scale=0.7,center}{

\begin{tabular}{|c|c|c|c|c|c|c|c|c|c|c|}
\hline
$r$ & $\mcD$ &  $\omega_{\sigma}$ 
& $R(\mathcal{O})$ & $R(\Oo^0)$ & $L/F$
& $\#X_{\nr}(M,\sigma)$ & $W_{\Oo}$ &$W_{\Oo^0}$ & $\mathcal{H}(G,\rho)$ & $\mathcal{H}(G^0,\rho^0)$\\ \hline\hline
\multirow{2}{*}{$r=0$}& \multirow{2}{*}{$((G,M),(y,\iota),(M_{y,0},\rho_M))$} & $=1$ 
   & $=1$ & $=1$
& unramified & $2$ 
& $\neq 1$ & $\neq 1$ & non-comm, $q_F$, $q_F$  &non-comm, $q_F$, $q_F$ 
\\ \cline{3-11}
&  & $\ne 1$ 
& $*$ & $*$  & unramified & 2 & $= 1$ & $= 1$ &$\mathbb{C}[R(\mathcal{O})]\ltimes\C[\Oo]$ &$\mathbb{C}[R(\mathcal{O})]\ltimes\C[\Oo]$
\\ \cline{3-11}
\hline
\multirow{2}{*}{$r\neq 0$}
& \multirow{2}{*}{$(((M,G),M),(y,\iota),(r,0),(\phi,1),(M_{y,0},\rho_M))$}
&$=1$ 
  & \multirow{2}{*}{$= 1$}
&\multirow{2}{*}{$= 1$}  & \multirow{2}{*}{unramified}
& \multirow{2}{*}{2} & \multirow{2}{*}{$=1$} &\multirow{2}{*}{$=1$} &\multirow{2}{*}{$\C[\mathcal{O}]$}&\multirow{2}{*}{$\C[\Oo^0]$}
\\ \cline{3-3}
& 
&$\neq 1$ 

  & 
&
&
&  &  & &&
\\ \hline
\end{tabular}}
\smallbreak
\begin{center} 
{\sc Table} \ref{num: M short 0}.
\end{center}
\end{numberedparagraph}

\begin{numberedparagraph}\textbf{Short root positive depth case.}\\
\noindent{{ (a)
$G^0=\rU_{\varpi'}(1,1)$ case:}} 
\begin{itemize}\item
When 
$\phi_0|_{\rZ_M^{0}}=1$, the Plancherel formula on the $\rG_2$ side in \ref{subsub: short root unramified} has a zero, and thus $W_{\mathcal{O}}\neq \{1\}$ and thus $R(\mathcal{O})=\{1\}$. In this case the Hecke algebra $\mathcal{H}(G,\rho)$ is non-commutative, and the $q$-parameter is just $q=q_F$. By Lemma \ref{lem:GG0}, we have $W(M^0,\Oo^0)\neq 1$, and since $W_{\Oo^0}\neq 1$ by \ref{subsub: Upi}, we have $R(\Oo^0)=1$. Moreover, the Hecke algebra $\mathcal{H}(G^0,\rho^0)$ has parameter $q_F$ by \ref{subsub: Upi}. 

\item When the central character (of $\mathrm{GL}_2^{\mathrm{short}}$) $\phi_0|_{Z_M^{0}}=\text{sign character}\neq 1$ is ramified, the Plancherel formula on the G2 side in \ref{subsub: short root ramified} has no zero, and thus $W_{\mathcal{O}}=\{1\}$. In this case the Hecke algebra $\mathcal{H}(G,\rho)=\mathbb{C}[R(\mathcal{O})]\ltimes\C[\Oo]$ has trivial $q$-parameter. On the other hand, since $I(\sigma)$ is reducible by \cite[Proposition 6.2]{Shahidi-Plancherel-Langlands}, we have  $R_{\sigma}\neq 1$. Since $W_{\sigma}\rtimes R_{\sigma}= W(M,\sigma)\subseteq W(M,\Oo)=W_{\mathcal{O}}\rtimes R(\Oo)$ and $W_{\mathcal{O}}=1$, we have $R_{\sigma}\subseteq R(\Oo)$ and thus $R(\Oo)\neq 1$. 
By Lemma \ref{lem:GG0}, we also have $R(\Oo^0)\neq 1$ since $W_{\Oo^0}=1$ by \ref{subsub: Upi}. 

\end{itemize}
\noindent{{ (b)
$G^0=\rU_{\epsilon}(1,1)$ case:}} 
 When 
 $\phi_0|_{\rZ_M^{0}}=1$ 
 the Plancherel formula on the $\rG_2$ side in \ref{subsub: short root unramified} has a zero, and thus $W_{\mathcal{O}}\neq \{1\}$ and thus $R(\mathcal{O})=\{1\}$. In this case the Hecke algebra $\mathcal{H}(G,\rho)$ is non-commutative, and the $q$-parameter is just $q=q_F$. 
 From \ref{subsub: Ue}, we have $W_{\Oo^0}\ne \{1\}$ and $R(\Oo^0)=\{1\}$.
\end{numberedparagraph}

\begin{numberedparagraph}\label{num: M short pos}\textbf{Table for short root positive depth cases.}~We summarize the above in the following table:

\adjustbox{scale=0.85,center}{

\begin{tabular}{|c|c|c|c|c|c|c|c|c|c|c|c|}
\hline
$M^0$ &  $\phi_0|_{\rZ_M^0}$ &$\phi_1$  &$\vec G$ & $R(\mathcal{O})$ & $R(\Oo^0)$ & $L/F$
& $\fX_{\nr}(M,\sigma)$ & $W_{\Oo}$ &$W_{\Oo^0}$ & $\mathcal{H}(G,\rho)$ & $\mathcal{H}(G^0,\rho^0)$\\ \hline\hline
\multirow{4}{*}{$T_{\alpha,\varpi'}$} & $=1$ &$=1$
& \multirow{3}{*}{$(U_{\varpi'}(1,1),G)$}   & $=1$ & $=1$
& ramified  & $1$ 
& $\neq 1$ & $\neq 1$ & non-comm, $q_F$  &non-comm, $q_F$
\\ \cline{2-3}\cline{5-12}
&$\begin{array}{c}
     =\text{ sign}\\
     \text{character}
\end{array}$&$\neq 1$&&$\neq 1$&$\neq 1$&ramified&1&$=1$&$=1$&$\mathbb{C}[R(\mathcal{O})]\ltimes\C[\Oo]$&$\mathbb{C}[R(\mathcal{O})]\ltimes\C[\Oo^0]$\\
\cline{2-12}
& $\begin{array}{c}\ne 1\\ \neq\text{sign character} \end{array}$ & $= 1$ & $(M^0,G)$  & $= 1$ &$= 1$  & ramified & $1$ & $=1$ & $=1$ &$\C[\Oo]$& $\C[\Oo^0]$
\\ \cline{2-12}
&both & $\neq 1$ & $\begin{array}{c}(M^0,M,G)\end{array}$ & $= 1$  & $= 1$ & ramified
 & 1 &$=1$
& $=1$  &$\C[\Oo]$ &$\C[\Oo^0]$
\\ \hline
\multirow{6}{*}{$T_{\alpha,\varepsilon}$}
& \multirow{2}{*}{$=1$} &\multirow{2}{*}{$=1$}
& \multirow{2}{*}{$(U_\varepsilon(1,1),G)$}  & \multirow{2}{*}{$=1$}  & \multirow{2}{*}{$=1$}
& \multirow{2}{*}{unramified} & \multirow{2}{*}{2}
& \multirow{2}{*}{$\neq 1$} & \multirow{2}{*}{$\neq 1$} &\multirow{2}{*}{$\begin{array}{c}\text{non-comm.}\\ q_F\end{array}$} &\multirow{2}{*}{$\begin{array}{c}\text{non-comm.}\\q_F\end{array}$}
\\ 
& & & &  &  & & & && &
\\ \cline{2-12}
& \multirow{2}{*}{$\ne 1$ } &\multirow{2}{*}{$=1$}
& \multirow{2}{*}{$(M^0,G)$}  & \multirow{2}{*}{$=1$} & \multirow{2}{*}{$=1$}
& \multirow{2}{*}{unramified}
& \multirow{2}{*}{2} & \multirow{2}{*}{$=1$} & \multirow{2}{*}{$=1$} &\multirow{2}{*}{$\C[\Oo]$} &\multirow{2}{*}{$\C[\Oo^0]$}
\\ 
& & & &  &  & & & && &
\\ \cline{2-12}
& \multirow{2}{*}{\ both} &\multirow{2}{*}{$\neq 1$}
&  \multirow{2}{*}{$\begin{array}{c}(M^0,M,G)\end{array}$} & \multirow{2}{*}{$=1$} & \multirow{2}{*}{$=1$}
& \multirow{2}{*}{unramified} 
& \multirow{2}{*}{2}  & \multirow{2}{*}{$=1$} 
& \multirow{2}{*}{$=1$} &\multirow{2}{*}{$\C[\Oo]$} &\multirow{2}{*}{$\C[\Oo^0]$} 
\\ 
& & & &  &  & & & && &
\\ \hline
\end{tabular}}
\smallbreak
\begin{center} {\sc Table} \ref{num: M short pos}.
\end{center}
\end{numberedparagraph}

\begin{numberedparagraph}
We keep the notations of \S\ref{subsec:general}. 
The following theorem establishes the validity, for $\rG_2$, of a generalization of a conjecture of Yu's \cite[Conjecture~0.2]{Yu} for supercuspidal types, which was proved by Ohara in \cite{Ohara}. The following result shows that a stronger version of Theorem \ref{thm:xiGG0}(2) holds for the group $\rG_2$.

\begin{theorem} \label{Weyl_iso} \
Let $p\ne 2,3$. The algebras $\mcH^\fs(G):=\End_G(\Pi^\fs_G)$ and $\mcH^{\fs^0}(G^0):=\End_{G^0}(\Pi^{\fs^0}_{G^0})$ are isomorphic.
\end{theorem}
\begin{proof} 
By Proposition~\ref{prop:comparisons}, it is equivalent to show that the algebras $\mcH(G,\rho_\mcD)$ and $\mcH(G^0,\rho_{\mcD^0})$ are isomorphic. The latter can be read directly from the tables \ref{num: M long 0}, \ref{num: M long pos}, \ref{num: M short 0} and \ref{num: M short pos}. 
\end{proof}
The following corollary is a stronger version of Lemma~\ref{lem:Ad-Mi} for  $G=\rG_2$. 
\begin{Coro}\label{RO-reduction-step}
The groups $R(\Oo)\simeq R(\Oo^0)$ and $W_{\Oo}\simeq W_{\Oo^0}$.
\end{Coro}
\begin{proof}
This can be read directly from our tables \ref{num: M long 0}, \ref{num: M long pos}, \ref{num: M short 0} and \ref{num: M short pos}, with explanations given in the sections immediately preceding the tables.
\end{proof}
\end{numberedparagraph}

\begin{numberedparagraph}\textbf{On Lusztig's conjecture.}~
Let $L^\fs\colon W_\aff^\fs\to\bbN$ be the weight function\footnote{i.e., $L^\fs(w)>0$ for all $w\in W_\aff^\fs-\{1\}$, and $L^\fs(ww')=L^\fs(w)+L^\fs(w')$ for any $w,w'\in W_\aff^\fs$ such that $\ell(ww')=\ell(w)+\ell(w')$.} on $W_\aff^\fs$ defined by
\begin{equation} \label{eqn:weight function}
L^\fs(s_\alpha):=\lambda(\alpha)\quad\text{and}\quad L^\fs(s'_\alpha):=\lambda^*(\alpha).
\end{equation}
In \cite[\S1.a]{Lusztig-Open-problems}, Lusztig made the following conjecture.  
\begin{conj}\label{Lusztig-conjecture-parameters} (Lusztig)
The  function $L^\fs$ on the affine Weyl group $W_\aff^\fs$ is in the collection of weight functions described in 
\cite{Lusztig-Kyoto, Lusztig-p-adic-I, Lusztig-p-adic-II}.
\end{conj}
Many cases of Conjecture~\ref{Lusztig-conjecture-parameters} have been proved in \cite{Solleveld-Hecke-algebra-2}, e.g.~for principal series representations of $G$. 
\begin{theorem} \label{thm:Lusztig-conj-parameters}
Conjecture~\ref{Lusztig-conjecture-parameters}
 holds for the group $\rG_2$.
\end{theorem}
\begin{proof}
It follows from Tables~\ref{num: M long 0}, \ref{num: M long pos}, \ref{num: M short 0} and \ref{num: M short pos}.
\end{proof}
\end{numberedparagraph}

\section{Applications to other groups}
\addtocontents{toc}{\protect\setcounter{tocdepth}{0}}
Let $N$ be a positive integer. Let $J'_N$ denote the $N\times N$-matrix $\left(\begin{smallmatrix}
&&&&&1\cr
&&&&1&\cr
&&&.&&\cr
&&.&&&\cr
&.&&&&\cr
1&&&&&
\end{smallmatrix}\right)$. When $N=2n$, let $J_N:=\left(\begin{smallmatrix}0&\II_n\cr -\II_n&0\end{smallmatrix}\right)$.
\subsection{Symplectic group}
The $F$-rational points of the symplectic group $\Sp_{2n}$ are given by
\begin{equation}
    \Sp_{2n}(F)=\left\{g\in\GL_{2n}(F)\,:\,{}^tgJ_{2n}g= J_{2n}\right\}.
\end{equation}
Let $P$ be the maximal parabolic subgroup of $\GSp_{2n}(F)$, consisting of matrices whose lower left $n\times n$-block is zero. The Levi factor of $P$ is isomorphic to $\GL_n(F)$.

\subsection{General symplectic group}
The $F$-rational points of the algebraic group $\GSp_{2n}$ are given by
\begin{equation}
    \GSp_{2n}(F)=\left\{g\in\GL_{2n}(F)\,:\,{}^tgJ_{2n}g=\mu_n(g) J_{2n},\mu_n(g)\in F^\times\right\}.
\end{equation}
Let $P$ be the maximal parabolic subgroup of $\GSp_{2n}(F)$, consisting of matrices whose lower left $n\times n$-block is zero. The Levi factor of $P$ is isomorphic to $\GL_n(F)\times\GL_1(F)$.

\subsection{Special orthogonal group}
The $F$-rational points of the algebraic group $\SO_{N}$ are given by
\begin{equation}
    \SO_{N}(F)=\left\{g\in\GL_N(F)\,:\,\text{${}^tgJ'_Ng= J'_N$, $\det(g)=1$}\right\}.
\end{equation}
Let $P$ be the maximal parabolic subgroup of $\SO_{N}(F)$, consisting of matrices whose lower left $N\times N$-block is zero. 
The Levi factor of $P$ is isomorphic to $\GL_n(F)$, where $N=2n+1$ or $N=2n$.

\medskip
The LLC for $\GL_n(F)$, established in \cite{Harris-Taylor, Henniart-LLC-GLn,Scholze-LLC}, shows that the $L$-packets are always singletons in this case.
Thus, by Proposition~\ref{prop:singleton}, 
the properties (1) and (2) are satisfied in the three cases above.
\bibliographystyle{amsalpha}
\bibliography{bibfile}

\end{document}